%% file: submission_arxiv.tex
\documentclass[10pt,a4paper]{article}

\usepackage{hyperref}
\usepackage{a4wide}
\usepackage{amsmath,amsthm,epsfig,amssymb,amsbsy}
\usepackage{tikz}
\usepackage{pgfplots}
\usepackage{float}
\usepackage{caption}
\usepackage{subcaption}
\usepackage{enumerate}
\usepackage{comment}
\usepackage{booktabs}       % professional-quality tables
\usepackage{amsfonts}       % blackboard math symbols
\usepackage{nicefrac}
\usepackage{xspace}
\usepackage{tikz}

\newtheorem{theorem}{Theorem}[section]
\newtheorem*{theorem*}{Theorem}
\newtheorem*{lemma*}{Lemma}
\newtheorem*{definition*}{Definition}
\newtheorem{corollary}[theorem]{Corollary}
\newtheorem{lemma}[theorem]{Lemma}
\newtheorem{proposition}[theorem]{Proposition}
\newtheorem{definition}[theorem]{Definition}
\newtheorem{remark}[theorem]{Remark}
\newtheorem{example}[theorem]{Example}
\makeatletter
\renewenvironment{proof}[1][\proofname]{\par
  \normalfont \topsep6\p@\@plus6\p@\relax
  \trivlist
  \item[\hskip\labelsep
        \itshape
    #1\@addpunct{.}]\ignorespaces
}{%
  \nolinebreak\endtrivlist\@endpefalse
}
\makeatother

\input{incl_commands.tex}

\newcommand{\mynewline}{}

\title{Bregman Proximal Mappings and Bregman--Moreau Envelopes under Relative Prox-Regularity}
\author{
Emanuel Laude\footnote{Department of Informatics, Technical University of Munich, Germany, {\tt emanuel.laude@tum.de}},
Peter Ochs\footnote{Department of Mathematics, Saarland University, Germany, {\tt ochs@math.uni-sb.de}},
and Daniel Cremers\footnote{Department of Informatics, Technical University of Munich, Germany, {\tt cremers@tum.de}}}
\date{December 6, 2019}

\begin{document}
%\newbibliography{main}
%\bibliographystyle{main}{unsrt}
\maketitle
\abstract{\input{incl_abstract.tex}}
~

\captionsetup{width=1\linewidth}
\captionsetup{font=small}

\input{incl_body.tex}

\subsection*{Acknowledgements:} \input{incl_acknowledgements.tex}
\newpage
\appendix
\input{incl_appendix.tex}

\clearpage
\bibliography{references}
\bibliographystyle{abbrv.bst}

\end{document}

%% file: incl_commands.tex
 \newtheorem{algorithm}{\textbf{\textup{Algorithm}}}

\DeclareMathOperator*{\argmin}{arg\,min}

\DeclareMathOperator{\gph}{gph}
\DeclareMathOperator{\dom}{dom}
\DeclareMathOperator{\epi}{epi}
\DeclareMathOperator{\intr}{int}

\DeclareMathOperator{\bdry}{bdry}

\DeclareMathOperator{\con}{con}
\DeclareMathOperator{\ran}{ran}

\DeclareMathOperator{\lev}{lev}
\DeclareMathOperator{\cl}{cl}
\DeclareMathOperator{\sign}{sign}

\newcommand{\bR}{\mathbb{R}}

\newcommand{\bN}{\mathbb{N}}

\newcommand{\exR}{\overline{\mathbb{R}}}

\newcommand{\scal}[2]{\langle#1,#2\rangle}
\newcommand{\setsep}{\,:\,}
\newcommand{\set}[1]{\left\lbrace#1\right\rbrace}
\newcommand{\map}[3]{#1\colon#2\to#3}
\newcommand{\norm}[2][]{\Vert #2 \Vert_{#1}}
\newcommand{\eps}{\epsilon}
\newcommand{\abs}[1]{|#1|}

\newcommand{\iprod}[2]{\langle#1,#2\rangle}

\newcommand{\lenv}[3]{\ensuremath{\,\overleftarrow{\operatorname{env}}_{#3}^{#1} #2}}
\newcommand{\lprox}[3]{\ensuremath{\,\overleftarrow{\operatorname{prox}}_{#3}^{#1} #2}}

\newcommand{\renv}[3]{\ensuremath{\,\overrightarrow{\operatorname{env}}_{#3}^{#1} #2}}
\newcommand{\rprox}[3]{\ensuremath{\,\overrightarrow{\operatorname{prox}}_{#3}^{#1} #2}}

\newcommand{\q}{v}
\newcommand{\x}{x}
\newcommand{\y}{y}
\newcommand{\w}{u}
\newcommand{\z}{z}

\newcommand{\D}[1][]{%
\ifthenelse{\equal{#1}{}}{D}{D_{#1}}%
}

\usepackage{setspace}
\usepackage{xspace}

\makeatletter
\let\save@mathaccent\mathaccent
\newcommand*\if@single[3]{%
  \setbox0\hbox{${\mathaccent"0362{#1}}^H$}%
  \setbox2\hbox{${\mathaccent"0362{\kern0pt#1}}^H$}%
  \ifdim\ht0=\ht2 #3\else #2\fi
  }
\newcommand*\rel@kern[1]{\kern#1\dimexpr\macc@kerna}
\newcommand*\widebar[1]{\@ifnextchar^{{\wide@bar{#1}{0}}}{\wide@bar{#1}{1}}}
\newcommand*\wide@bar[2]{\if@single{#1}{\wide@bar@{#1}{#2}{1}}{\wide@bar@{#1}{#2}{2}}}
\newcommand*\wide@bar@[3]{%
  \begingroup
  \def\mathaccent##1##2{%
    \let\mathaccent\save@mathaccent
    \if#32 \let\macc@nucleus\first@char \fi
    \setbox\z@\hbox{$\macc@style{\macc@nucleus}_{}$}%
    \setbox\tw@\hbox{$\macc@style{\macc@nucleus}{}_{}$}%
    \dimen@\wd\tw@
    \advance\dimen@-\wd\z@
    \divide\dimen@ 3
    \@tempdima\wd\tw@
    \advance\@tempdima-\scriptspace
    \divide\@tempdima 10
    \advance\dimen@-\@tempdima
    \ifdim\dimen@>\z@ \dimen@0pt\fi
    \rel@kern{0.6}\kern-\dimen@
    \if#31
      \overline{\rel@kern{-0.6}\kern\dimen@\macc@nucleus\rel@kern{0.4}\kern\dimen@}%
      \advance\dimen@0.4\dimexpr\macc@kerna
      \let\final@kern#2%
      \ifdim\dimen@<\z@ \let\final@kern1\fi
      \if\final@kern1 \kern-\dimen@\fi
    \else
      \overline{\rel@kern{-0.6}\kern\dimen@#1}%
    \fi
  }%
  \macc@depth\@ne
  \let\math@bgroup\@empty \let\math@egroup\macc@set@skewchar
  \mathsurround\z@ \frozen@everymath{\mathgroup\macc@group\relax}%
  \macc@set@skewchar\relax
  \let\mathaccentV\macc@nested@a
  \if#31
    \macc@nested@a\relax111{#1}%
  \else
    \def\gobble@till@marker##1\endmarker{}%
    \futurelet\first@char\gobble@till@marker#1\endmarker
    \ifcat\noexpand\first@char A\else
      \def\first@char{}%
    \fi
    \macc@nested@a\relax111{\first@char}%
  \fi
  \endgroup
}
\makeatother

%% file: incl_abstract.tex
We systematically study the local single-valuedness of the Bregman proximal mapping and local smoothness of the Bregman--Moreau envelope of a nonconvex function under relative prox-regularity - an extension of prox-regularity - which was originally introduced by Poliquin and Rockafellar. 
As Bregman distances are asymmetric in general, in accordance with Bauschke et al., it is natural to consider two variants of the Bregman proximal mapping, which, depending on the order of the arguments, are called left and right Bregman proximal mapping. We consider the left Bregman proximal mapping first. Then, via translation result, we obtain analogue (and partially sharp) results for the right Bregman proximal mapping. The class of relatively prox-regular functions significantly extends the recently considered class of relatively hypoconvex functions. In particular, relative prox-regularity allows for functions with a possibly nonconvex domain. Moreover, as a main source of examples and analogously to the classical setting, we introduce relatively amenable functions, i.e. convexly composite functions, for which the inner nonlinear mapping is component-wise smooth adaptable, a recently introduced extension of Lipschitz differentiability. By way of example, we apply our theory to locally interpret joint alternating Bregman minimization with proximal regularization as a Bregman proximal gradient algorithm, applied to a smooth adaptable function.

%% file: incl_body.tex
\section{Introduction}
\subsection{Motivation and Related Work}
The Moreau envelope \cite{Moreau65} is a widely used and powerful tool in variational analysis and optimization, whose systematic study was initiated by Attouch \cite{Attouch77,Attouch84}. Given an input function, it yields a regularized function with several favorable properties such as differentiability and full domain. A key result, besides the aforementioned differentiability of the Moreau envelope, is the single-valuedness of the proximal mapping and a formula that relates its gradient to the proximal mapping.

For nonconvex functions, in general, these desirable properties are lost. However, if we focus solely on local properties, most of the results can be transferred to a certain class of nonconvex functions, namely prox-regular functions. Prox-regularity was introduced by Poliquin and Rockafellar \cite{PoRo96} and comprises several widely used classes of functions, such as primal-lower-nice functions \cite{Poliquin91}, subsmooth functions, strongly amenable functions \cite{PoRo96,RoWe98}, and proper lower semi-continuous convex functions. Prox-regular functions often behave locally like convex functions and, therefore, the single-valuedness property of the proximal mapping and the gradient formula for the Moreau envelope of a prox-regular function are locally valid \cite{PoRo96} (see \cite{BBEM10,JTZ14} for the infinite dimensional setting). In particular, indicator functions of closed and prox-regular sets have rich properties \cite{poliquin2000local}. \\

While these concepts are closely tied to Euclidean geometry, Bregman \cite{Bre67} introduced a generalization of the Euclidean distance, which has proved very effective for optimization. In particular, the Bregman-proximal mapping is a key component in several algorithms (for example, see \cite{BBT2016,BSTV2018,Teboulle92,bauschke2019linear,MOPS19,bauschke2006joint,bauschke2003bregman,BC03,Eckstein93,OFB18,BC01,CR98,CH02,CZ97,KRS11,Kiwiel1997,Nguyen16,davis2018stochastic,RRP18,HRX18,LFN18,BK12,MO19b,Nguyen17,BBES17,CZ92,bauschke2017regularizing} and references therein) and is known to improve constants in convergence rate estimates \cite{NY83,Teboulle92}. Recently, there has been increasing interest in understanding the regularization properties for Bregman--Moreau envelopes \cite{bauschke2017regularizing,chen2012moreau,kan2012moreau,bauschke2009bregman,davis2018stochastic}, generalizing the known results for the Moreau envelope. While \cite{bauschke2009bregman,bauschke2017regularizing} consider a fully convex setting with jointly convex Bregman distances, the aforementioned properties also hold for certain nonconvex functions: For so-called (relatively) hypoconvex functions, i.e. functions that become convex by adding a distance-generating kernel (or Legendre) function \cite{chen2012moreau,kan2012moreau,davis2018stochastic}, single-valuedness and the gradient formula of the Bregman--Moreau envelope are proved in \cite{chen2012moreau,kan2012moreau}. This generalizes the Euclidean distance setting; see e.g. \cite{wang2010chebyshev}.

However, hypoconvexity is a global property, which is not satisfied by many nonconvex functions in the Euclidean setting that are still prox-regular. In fact, hypoconvex functions must have a convex domain. More specifically, for an indicator function of a closed set, at least in finite dimensions, the global single-valuedness of the Bregman proximal mapping is equivalent to the convexity of the set \cite{bauschke2009bregman}. In a Euclidean setting, this is also known as the Chebyshev problem.

In this paper, we introduce the counterpart of prox-regular functions with respect to a nonlinear geometry induced by a Bregman distance. We refer to them as \emph{relatively prox-regular} or, to make the geometry explicit, prox-regular relative to a distance-generating Legendre function. We systematically study this novel class of functions, provide several examples, and generalize results for Moreau envelopes of (classically) prox-regular functions to Bregman--Moreau envelopes of relatively prox-regular functions. 

Intuitively, classically prox-regular functions allow for a local lower-quadratic support. However, obviously, several simple functions cannot be supported by quadratic functions, and even if they may be supported by a quadratic function, a tighter approximation may be obtained for another (e.g. higher degree polynomial) supporting function. We exploit the improved lower approximation and derive generalizations of the results that are known from the Euclidean setting, such as local single-valuedness of the Bregman proximal mapping, differentiability and the gradient formula for the Bregman--Moreau envelope. Although several functions that are relatively prox-regular, are also classically prox-regular, the classical theory cannot explain all of these situations. 

Moreover, analogously to the Euclidean setting, for which strongly amenable functions provide a source of examples for prox-regular functions, we introduce relatively amenable functions. Their definition is based on a recent generalization of functions that have a Lipschitz continuous gradient to smooth adaptable (or relatively smooth) functions \cite{BBT2016,BSTV2018}, i.e. functions that are continuously differentiable and convex relative to a distance-generating Legendre function.\\%, which is a concept that is also adjusted to Bregman geometry. \\

While our main focus is the (left) Bregman proximal mapping for such relatively prox-regular nonconvex functions, we shall also transfer our results to the right Bregman proximal mapping, that was introduced in \cite{bauschke2006joint} and further studied in \cite{bauschke2017regularizing} for convex functions and jointly convex Bregman distances.

In addition, we provide explicit formulas for the gradients of both the left and the right Bregman--Moreau envelope that hold locally, for example, in a neighborhood of a limit point of an algorithm. In practice, this allows us to interpret a stationary point of a joint minimization problem as a stationary point of a smoothed model involving a Bregman--Moreau envelope. Such a ``translation of stationarity'' has been observed previously in \cite{LWC18,laude-wu-cremers-aistats-19} for the classical Moreau envelope and an anisotropic generalization of the former, both under prox-regularity. In addition, we apply our theory to locally interpret joint alternating Bregman minimization with proximal regularization as a Bregman proximal gradient algorithm.

\subsection{Outline and Summary of our Contribution}
The remainder of the paper is organized as follows: Section~\ref{sec:preliminaries} clarifies the notation and collects important properties and facts about \emph{Bregman} distances generated by \emph{Legendre} functions, also called (distance-generating) \emph{kernel} functions.

In Section~\ref{sec:continuity}, we recall the definition of \emph{relative prox-boundedness} \cite[Definition 2.3]{kan2012moreau} and state some important results about the continuity properties of the nonconvex Bregman proximal mapping and associated envelope function from \cite{kan2012moreau}. We complement these results by formulating equivalent characterizations of relative prox-boundedness, and transfer the results also to the (nonconvex) right Bregman proximal mapping.

In Section~\ref{sec:relative_prox_regularity}, we study the \emph{single-valuedness of the left and right Bregman proximal mapping} of a \emph{relatively prox-regular} (nonconvex) function, a generalization of \cite[Definition 13.27]{RoWe98}, to a Bregman distance setting. In this context, we clarify the equivalence to classical prox-regularity under \emph{very strict convexity} of the kernel function.

In Section~\ref{sec:examples_prox_regular}, we identify \emph{relatively amenable} functions, i.e. the composition of a convex function and a smooth adaptable mapping \cite{BBT2016,BSTV2018}, as a main source of examples for relative prox-regularity. Relatively amenable functions generalize the notion of strongly amenable functions \cite[Definition 10.23]{RoWe98} and are in general not relatively hypoconvex. 
We also encounter the property of amenability when statements about the single-valuedness of the left Bregman proximal mapping are transferred to the right Bregman proximal mapping. This includes the single-valuedness of the right Bregman proximal mapping of a convex function, which yields a nonconvex minimization problem in general.

In Section~\ref{sec:gradient_formulas}, we study local regularity properties of the Bregman--Moreau envelope based on the local single-valuedness of the Bregman proximal mappings. This yields \emph{explicit formulas for the gradient of the left and right Bregman--Moreau envelope} that will hold locally.

As an example of using the developed theory in optimization, in Section~\ref{sec:algorithm}, we provide an example of an analytically solvable Bregman proximal mapping of a nonconvex relatively prox-regular function. In addition, we consider an \emph{alternating minimization algorithm with Bregman proximal regularization}. For relatively prox-regular functions, our results guarantee stationarity with respect to the equivalent inf-projected problem that involves the Bregman--Moreau envelope. Based on the gradient formulas for the Bregman--Moreau envelope, we conclude this section by highlighting a \emph{local equivalence between alternating minimization and recent Bregman proximal gradient algorithms} \cite{BBT2016,BSTV2018,bauschke2019linear}.

\section{Preliminaries and Notation} \label{sec:preliminaries}
Unless otherwise specified, we deploy the notation used in \cite{RoWe98}. In particular we adopt the notation of the regular $\widehat{\partial} f$, limiting $\partial f$ and horizon subgradients $\partial^\infty f$, from \cite[Definition 8.3]{RoWe98} and denote the indicator function of a set $C$ as $\delta_C$, i.e. $\delta_C(\x)=0$ if $\x\in C$ and $\delta_C(\x)=\infty$ otherwise. We say an extended real-valued function $f:\bR^m \to \exR:=\bR\cup\{-\infty, \infty\}$ is coercive if $f(\x) \to \infty$ for $\|\x\|\to \infty$ and super-coercive if $f(\x)/\|\x\| \to \infty$ for $\|\x\|\to \infty$.  According to \cite[Definition 1.33]{RoWe98}, a function $f:\bR^m \to \exR$ is locally lsc at $\bar \x$, a point where $f(\bar \x)$ is finite, if there is an $\epsilon > 0$, such that all sets of the form \mynewline $\{\x \in \bR^m : \|\x - \bar \x\| \leq \epsilon, f(\x) \leq \alpha\}$
with $\alpha \leq f(\bar \x) + \epsilon$ are closed.
For a $\mathcal{C}^1$ mapping $F:\bR^n \to \bR^m$ let $\nabla F(\x) \in \bR^{m \times n}$ denote the Jacobian of $F$ at $\x \in \bR^n$. With some slight abuse of notation, for a given set $C\subset \bR^m$, let $\con C$ denote its convex hull, i.e. the smallest convex set that contains $C$, and for an extended real-valued function $f$ let $\con f$ denote the largest convex function that is majorized by $f$.
Let $\Gamma_0(X)$ denote the set of all proper, lsc convex functions that map from some Euclidean space $X$ to $\exR$.
%Let $f^*$ denote the convex conjugate of $f$, defined by 
%\begin{align}
%f^*(\x):=\sup_{\y\in \bR^m} \langle \x, \y \rangle - f(\y).
%\end{align} 

A function $\phi \in \Gamma_0(X)$ of Legendre type is defined according to \cite[Section 26]{Roc70}:
\begin{definition}[Legendre function] \label{def:legendre}
The function $\phi \in \Gamma_0(\bR^m)$ is
\begin{enumerate}
\item[\rm (i)] \emph{essentially smooth}, if $\intr(\dom \phi) \neq \emptyset$ and $\phi$ is differentiable on $\intr(\dom \phi)$ such that $\|\nabla\phi(\x^\nu)\|\to \infty$, whenever $\x^\nu \to \x \in \bdry \dom \phi$, and
\item[\rm (ii)] \emph{essentially strictly convex}, if $\phi$ is strictly convex on every convex subset of $\dom \partial \phi$, and
\item[\rm (iii)] of \emph{Legendre} type, if $\phi$ is both essentially smooth and essentially strictly convex.
\end{enumerate}
\end{definition}

We list some basic properties of Legendre functions:
\begin{lemma} \label{lem:legendre_props}
Let $\phi \in \Gamma_0(\bR^m)$ of Legendre type. Then $\phi$ has the following properties:
\begin{enumerate}
\item[\rm (i)] $\dom \partial \phi = \intr (\dom\phi)$, {\rm\cite[Theorem 26.1]{Roc70}}. 
\item[\rm (ii)] $\phi^*$ is of Legendre type, {\rm\cite[Theorem 26.5]{Roc70}}.
\item[\rm (iii)] $\map{\nabla \phi}{\intr (\dom\phi)}{\intr (\dom\phi^*)}$ is bijective with inverse \mynewline $\map{\nabla \phi^*}{\intr (\dom\phi^*)}{\intr (\dom\phi)}$ with both $\nabla \phi$ and $\nabla \phi^*$ continuous on $\intr (\dom\phi)$ resp. $ \intr(\dom\phi^*)$, {\rm\cite[Theorem 26.5]{Roc70}}.
\item[\rm (iv)] $\phi$ is super-coercive if and only if $\dom \phi^*= \bR^m$, {\rm\cite[Proposition 2.16]{bauschke1997legendre}}.
\end{enumerate}
\end{lemma}

Even though our main focus is on general Legendre functions, many classical Legendre functions satisfy the following additional property, which we adopt from \cite[Definition 2.8]{bauschke2000dykstras}:
\begin{definition}[Very strictly convex functions]
Suppose $\phi \in \Gamma_0(\bR^m)$ is $\mathcal{C}^2$ on $\intr(\dom \phi) \neq\emptyset$ and $\nabla^2 \phi(x)$ is positive definite for all $x\in \intr(\dom \phi)$. Then we say $\phi$ is \emph{very strictly convex}.
\end{definition}

\begin{lemma} \label{lem:inverse_hessian}
Let $\phi \in \Gamma_0(\bR^m)$ be Legendre and very strictly convex. Then $\phi^*$ is Legendre and very strictly convex. Moreover, for any conjugate pair $\x \in \intr(\dom\phi)$ and $\nabla \phi(\x) \in \intr(\dom\phi^*)$ the Hessian matrices $\nabla^2 \phi(\x)$ and $\nabla^2 \phi^*(\nabla\phi(\x))$ are inverse to each other.
\end{lemma}
\begin{proof}
By Lemma~\ref{lem:legendre_props} we know that $\phi^*\in \Gamma_0(\bR^m)$ is Legendre. By assumption $\map{\nabla \phi}{\intr (\dom\phi)}{\intr (\dom\phi^*)}$ is continuously differentiable on $\intr(\dom \phi)$ with derivative $\nabla^2 \phi(\x)$ invertible for any $\x \in \intr(\dom \phi)$. Thus, by the inverse function theorem for any $\x \in \intr(\dom \phi)$ there exist open neighborhoods $V$ of $\x$ and $U$ of $\nabla \phi(\x)$ such that locally $\map{(\nabla \phi)^{-1}}{U}{V}$ is continuously differentiable with derivative $\nabla ((\nabla \phi)^{-1}) (\nabla\phi(\x)) = (\nabla^2 \phi(\x))^{-1}$.
Since $(\nabla \phi)^{-1} = \nabla \phi^*$ and $(\nabla^2 \phi(\x))^{-1}$ is positive definite, the assertion follows.
\qed
\end{proof}

For examples of typical Legendre functions (e.g. Boltzmann--Shannon, Burg's or Fermi-Dirac entropy, Hellinger, Fractional Power) as well as their convex conjugates and derivatives we refer to \cite[Example 2.2]{bauschke2019linear}. More examples can be found in \cite{Bre67,Teboulle92,Eckstein93,bauschke1997legendre,bauschke2001essential}. In particular, we highlight that the Legendre function $\phi(\x)=(1/p)|\x|^p$, $p>1$, is not very strictly convex for $p\neq 2$ and not even $\mathcal{C}^2$ if $p\in\left]1,2\right[$. The class of Legendre functions induces favorable properties for the following generalized distance-like measure.
\begin{definition}[Bregman distance] \label{def:bregman_dist}
Let $\phi \in \Gamma_0(\bR^m)$ be Legendre. Then, the Bregman distance $\map{\D[\phi]}{\bR^m \times \bR^m}{\exR}$ generated by the kernel $\phi$ is defined by
\begin{align}
\D[\phi](\x,\y) =  \begin{cases}\phi(\x)- \phi(\y)-\iprod{\nabla \phi(\y)}{\x-\y}, & \text{if $\y \in \intr (\dom \phi)$,} \\
+\infty, & \text{otherwise.}
\end{cases}
\end{align}
\end{definition}
\begin{lemma} \label{lem:bregman_props}
Let $\phi \in \Gamma_0(\bR^m)$ be Legendre. Then the following properties hold for the Bregman distance $\D[\phi](\cdot,\cdot)$ induced by $\phi$:
\begin{enumerate}
\item[\rm (i)] For all $\x \in \bR^m$ and $\y \in \intr(\dom \phi)$ we have $\D[\phi](\x,\y) = 0 \iff \x = \y$, {\rm\cite[Theorem 3.7(iv)]{bauschke1997legendre}}.
\item[\rm (ii)] For all $\x,\y \in \intr(\dom \phi)$ we have $\D[\phi](\x,\y) = \D[\phi^*](\nabla \phi(\y), \nabla \phi(\x))$, {\rm\cite[Theorem 3.7(v)]{bauschke1997legendre}}.
\item[\rm (iii)] If $\phi$ is very strictly convex for any compact and convex $K \subset \intr(\dom \phi)$ there exist positive scalars $\Theta,\theta>0$ such that
$$
\frac{\theta}{2}\|\x-\y\|^2 \leq \D[\phi](\x,\y) \leq \frac{\Theta}{2}\|\x-\y\|^2,
$$
for any $\x, \y \in K$, {\rm\cite[Proposition 2.10]{bauschke2000dykstras}}.
\end{enumerate}
\end{lemma}

It should be noted that for the remainder of this paper, we have not made standing assumptions. Instead, we explicitly state the assumptions in each theorem separately.

\section{Bregman Proximal Mappings and Moreau Envelopes} \label{sec:bregman_prox}
\subsection{Definition, Properness and Continuity} \label{sec:continuity}
We define the left Bregman--Moreau envelope and proximal mapping with step-size parameter $\lambda>0$ according to \cite{kan2012moreau} or \cite{bauschke2017regularizing} for the convex setting.
\begin{definition}[Left Bregman--Moreau envelope and proximal mapping]
Let $\phi \in \Gamma_0(\bR^m)$ be Legendre and $\map{f}{\bR^m}{\exR}$ be proper. For some $\lambda > 0$ and $\y \in \bR^m$ we define the left Bregman--Moreau envelope (in short: left envelope) of $f$ at $\y$ as
\begin{align} \label{eq:left_env}
\lenv{\phi}{f}{\lambda}(\y)&=\inf_{\x \in \bR^m} f(\x) + \frac{1}{\lambda} \D[\phi](\x,\y),
\end{align}
and the associated left Bregman proximal mapping (in short: left prox) of $f$ at $\y$ as
\begin{align} \label{eq:left_prox}
\lprox{\phi}{f}{\lambda}(\y)=\argmin_{\x \in \bR^m} f(\x) + \frac{1}{\lambda} \D[\phi](\x,\y).
\end{align}
\end{definition}
From the definition, it is clear that %$\lprox{\phi}{f}{\lambda}$ can be rewritten as
%\begin{align} \label{eq:prox_set}
%\lprox{\phi}{f}{\lambda}(\y) = \left\{ \x \in \dom f \cap \dom \phi : f(\x) + \frac{1}{\lambda} \D[\phi](\x,\y) = \inf_{\x' \in \bR^m} f(\x') + \frac{1}{\lambda} \D[\phi](\x',\y)<\infty \right\} \,.
%\end{align}
$\dom(\lprox{\phi}{f}{\lambda}) \subset \intr (\dom \phi)$ and \mynewline$\dom(\lenv{\phi}{f}{\lambda}) \subset \intr (\dom \phi)$.

The set $\lprox{\phi}{f}{\lambda}(\y)$ is possibly empty in the nonconvex setting. A sufficient condition which guarantees that the Bregman proximal mapping is non-empty is prox-boundedness, which we adapt from \cite[Definition 2.3]{kan2012moreau}.
\begin{definition}[Relative prox-boundedness]
Let $\phi \in \Gamma_0(\bR^m)$ be Legendre and $\map{f}{\bR^m}{\exR}$ be proper. We say $f$ is prox-bounded relative to $\phi$ if there exists $\lambda > 0$ such that $\lenv{\phi}{f}{\lambda}(\y)> -\infty$ for some $\y\in \intr(\dom \phi)$. The supremum of the set of all such $\lambda$ is the threshold $\lambda_f$ of the prox-boundedness, i.e.
$$
\lambda_f = \sup \big\{ \lambda > 0 : \exists\, \y \in \intr(\dom \phi) : \lenv{\phi}{f}{\lambda}(\y)> -\infty \big\}.
$$
\end{definition}

Prox-boundedness also allows us to extract a continuity property for both the Bregman proximal mapping and the Bregman--Moreau envelope. The following result summarizes important properties of the left envelope from~\cite{kan2012moreau}.

\begin{lemma}[Continuity properties of the left prox and envelope]  \label{lem:continuity_left}
Let $\phi \in \Gamma_0(\bR^m)$ be Legendre and super-coercive and $\map{f}{\bR^m}{\exR}$ be proper, lsc and relatively prox-bounded with threshold $\lambda_f$ and let $\lambda \in \left]0,\lambda_f\right[$. Assume that $\dom \phi \cap \dom f \neq \emptyset$. Then $\lprox{\phi}{f}{\lambda}$ and $\lenv{\phi}{f}{\lambda}$ have the following properties:
\begin{enumerate}
\item[\rm (i)] $\lprox{\phi}{f}{\lambda}(\y) \neq \emptyset$ is compact for all $\y \in \intr (\dom \phi)$ and the envelope $\lenv{\phi}{f}{\lambda}$ is proper, {\rm\cite[Theorem 2.2(i)]{kan2012moreau}}.
\item[\rm (ii)] The envelope $\lenv{\phi}{f}{\lambda}$ is continuous on $\intr (\dom \phi)$, {\rm\cite[Corollary 2.2]{kan2012moreau}}.
\item[\rm (iii)] For any sequence $\y^\nu \to \y^* \in \intr (\dom \phi)$ and $\x^\nu \in \lprox{\phi}{f}{\lambda}(\y^\nu)$ we have $\{\x^\nu\}_{\nu \in \bN}$ is bounded and all its cluster points $\x^*$ lie in $\lprox{\phi}{f}{\lambda}(\y^*)$, {\rm\cite[Corollary 2.4]{kan2012moreau}}.
\end{enumerate}
\end{lemma}
Note that in general the left Bregman--Moreau envelope is not lsc relative to $\bR^m$, cf. \cite[Remark 3.6]{bauschke2006joint}.

We complement the results of \cite{kan2012moreau} by stating equivalent characterizations of relative prox-boundedness.
To this end, we need the following lemma, which is analogous to \cite[Exercise 1.14]{RoWe98}.
\begin{lemma} \label{lem:relative_growth_equivalence}
Let $\map{\phi}{\bR^m}{\exR}$ be proper, lsc and coercive with $\dom \phi = \bR^m$ and let $\map{f}{\bR^m}{\exR}$ be proper and lsc. Then we have the identity
$$
\liminf_{\|\x\|\to\infty} \frac{f(\x)}{\phi(\x)} = \sup \set{ \gamma\in \bR \setsep \exists \, \beta \in \bR \ \text{with}\ f(\x) \geq \gamma\phi(\x) + \beta\ \text{for all}\ \x\in \bR^m}.
$$
\end{lemma}
\begin{proof}
Note that since $\phi$ is coercive, we have $\phi(\x) \to \infty$, whenever $\|x\| \to \infty$. Let $\bar \gamma := \liminf_{\|\x\|\to\infty} f(\x)/\phi(\x)$ and 
$$
\gamma\in\Gamma:=\set{ \gamma\in \bR \setsep \exists \, \beta \in \bR \ \text{with}\ f(\x) \geq \gamma\phi(\x) + \beta\ \text{for all}\ \x\in \bR^m}.
$$
This means that there exists $\beta$ such that $f(\x) \geq \gamma\phi(\x) + \beta$, for all $\x\in \bR^m$. Dividing by $\phi(\x) > 0$ which holds for $\|\x\| >t$, for some $t>0$, and taking the $\liminf$ on both sides yields $\liminf_{\|\x\|\to\infty} f(\x)/\phi(\x) \geq \gamma + 0$, meaning that $\bar \gamma \geq \gamma$.
Now let $\gamma \in \bR$ with $-\infty <\gamma < \bar \gamma \leq +\infty$ be finite. Suppose that for any compact level set $\emptyset \neq \lev_{\leq r} \phi:= \{\x \in \bR^m : \phi(\x) \leq r \}$ with $r > 0$ there exists $\x \in (\lev_{\leq r} \phi)^c \neq \emptyset$ in the complement of $\lev_{\leq r} \phi$, which is non-empty due to the coercivity of $\phi$, with $f(\x) < \gamma \phi(\x)$. In particular this means that there is a sequence $\x^\nu \in \bR^m$ with $\phi(\x^\nu) \to \infty$ and $f(\x^\nu)/\phi(\x^\nu) < \gamma$. Taking the $\liminf$ on both sides implies due to the coercivity of $\phi$ that
$$
\bar \gamma = \lim_{r\to \infty} \left( \inf_{r < \|\x\|} \frac{f(\x)}{\phi(\x)} \right) \leq \lim_{\nu\to \infty} \frac{f(\x^\nu)}{\phi(\x^\nu)} \leq \gamma < \bar \gamma,
$$
which is a contradiction. This means that there is $r > 0$ such that for any $\x \in ( \lev_{\leq r} \phi)^c \neq \emptyset$ we have
$
f(\x) \geq \gamma \phi(\x).
$
By assumption $h:=f - \gamma \phi$ is proper lsc. In view of \cite[Corollary 1.10]{RoWe98}, $h$ is bounded below on $\lev_{\leq r} \phi$, showing that for some $\beta \in \bR$ sufficiently small, it holds $f(\x) \geq \gamma \phi(\x) + \beta$ for any $\x \in \lev_{\leq r} \phi$. Overall we have
$
f(\x) \geq \gamma \phi(\x) + \beta,
$
for all $\x \in \bR^m$. This shows the result.
\qed
\end{proof}

The following proposition adapts \cite[Exercise 1.24]{RoWe98} to a Bregman distance setting.
\begin{proposition}[Characterization of relative prox-boundedness] \label{prop:relative_prox_bounded_equivalence}
  Let $\phi\in \Gamma_0(\bR^m)$ be Legendre and super-coercive and let $\map{f}{\bR^m}{\exR}$ be proper and lsc with $\dom f \cap \dom \phi \neq \emptyset$. Then, the following properties are equivalent:
  \begin{enumerate}
    \item[\rm(i)]\label{prop:relative_prox_bounded_equivalence_A} $f$ is prox-bounded relative to $\phi$.
    \item[\rm(ii)]\label{prop:relative_prox_bounded_equivalence_B} for some $r>0$ the function $f+r \phi$ is bounded from below on $\bR^m$.
    \end{enumerate}
If futhermore $\dom \phi = \bR^m$ the above properties are equivalent to 
\begin{align}\label{eq:relative_prox_bounded_equivalence_C} 
\liminf_{\|\x\|\to\infty} \frac{f(\x)}{\phi(\x)}> - \infty.
\end{align}
\end{proposition}
\begin{proof}
The equivalence between (i) and (ii) follows thanks to the properness of the envelope function from Lemma~\ref{lem:continuity_left}: 

(i) $\implies$ (ii): Let $f$ be prox-bounded relative to $\phi$ with threshold $\lambda_f > 0$. In view of Lemma~\ref{lem:continuity_left}(ii), we have for any $\lambda \in \left]0,\lambda_f\right[$ that 
$$
{\lenv{\phi}{f}{\lambda}}{\left(\nabla \phi^*(0)\right)} = \inf_{\x \in \bR^m} f(\x) + \frac{1}{\lambda} \phi(\x) -  \frac{1}{\lambda} \phi(\nabla \phi^*(0)) > -\infty.
$$
(ii) $\implies$ (i): Let $r>0$. Then there exists $\beta \in \bR$ such that $f(\x) + r \phi(\x) \geq \beta$ for any $\x \in \bR^m$. Adding $-r\phi(\nabla \phi^*(0))$ to both sides of the inequality yields
$$
f(\x) + r\D[\phi](\x, \nabla \phi^*(0)) \geq \beta -r\phi(\nabla \phi^*(0)),
$$
for all $\x \in \bR^m$ and the assertion follows for $\y := \nabla \phi^*(0)$ and $\lambda:= 1/r$.

To show the remaining statement, assume that $\dom \phi=\bR^m $ and let \eqref{eq:relative_prox_bounded_equivalence_C} hold. In view of Lemma~\ref{lem:relative_growth_equivalence}, we have that
$$
\sup \set{ \gamma\in \bR \setsep \exists \, \beta \in \bR \ \text{with}\ f(\x) \geq \gamma\phi(\x) + \beta\ \text{for all}\ \x\in \bR^m} > -\infty.
$$
Then there exists a finite $+\infty > \gamma>-\infty$ such that $f(x) \geq \gamma\phi(x) + \beta$ holds for some $\beta\in\bR$ and any $\x\in \bR^m$. For $r>\max\{0,-\gamma\}$, we have that $r+\gamma\geq 0$ and
$$
 f+r\phi \geq (r+\gamma) \phi+\beta>-\infty,
$$
since $\phi \in \Gamma_0(\bR^m)$ is coercive and therefore bounded from below on $\bR^m$, meaning we have~(ii).
  
Assume (ii) holds. By assumption there is some $\beta \in \bR^m$ such that for any $\x \in \bR^m$ we have:
$$
f(\x) > -r\phi(\x) + \beta.
$$
Let $\Gamma := \set{ \gamma\in \bR \setsep \exists \, \beta \in \bR \ \text{with}\ f(\x) \geq \gamma\phi(\x) + \beta\ \text{for all}\ \x\in \bR^m}$. \mynewline
Then $-r \in \Gamma$ and in view of Lemma~\ref{lem:relative_growth_equivalence}, we have $\liminf_{\|\x\|\to\infty} f(\x)/\phi(\x) =\sup \Gamma > -\infty$.
\qed
\end{proof}

For using the left Bregman proximal mapping in an algorithm, well-definedness is crucial, i.e. the output of one iteration must be compatible with the input of the next iteration. Usually, this can be achieved by the property 
$$
{\ran }\big(\lprox{\phi}{f}{\lambda}\big) \subset \intr (\dom \phi),
$$
which, however, requires a \emph{constraint qualification} (CQ):
\begin{lemma} \label{lem:prox_interior}
Let $\lambda > 0$, $\phi \in \Gamma_0(\bR^m)$ be Legendre and $f:\bR^m \to \exR$ be proper lsc. Assume that $\dom \phi \cap \dom f \neq \emptyset$ and that the following constraint qualification holds: 
\begin{align} \label{eq:qualification_interior}
\partial^\infty f(\x) \cap -N_{\dom \phi}(\x)=\{0\},
\end{align}
for any $\x \in \dom f \cap \dom \phi$. Then we have that %\footnote{Note that $\partial(\phi + \lambda f)^{-1}$ means $(\partial(\phi + \lambda f))^{-1}$.}
\begin{align}
{\ran} {\big(\lprox{\phi}{f}{\lambda}\big)} &\subset {\ran} {\left((\partial(\phi + \lambda f))^{-1}\circ \nabla \phi\right)} \\
&\subset {\ran} {\left((\partial(\phi + \lambda f))^{-1}\right)} \subset \intr (\dom \phi).
\end{align}
\end{lemma}
\begin{proof}
In case $\y \notin \intr (\dom \phi)$ we have $\lprox{\phi}{f}{\lambda} = \emptyset$ such that for the first inclusion only vectors $\y$ contained in $\intr (\dom \phi)$ matter: Fix $\y \in \intr (\dom \phi)$. By the definition of the left prox it is clear that $\ran (\lprox{\phi}{f}{\lambda}) \subset \dom f \cap \dom \phi$. For $\x \in \dom f \cap \dom \phi$, using \cite[Corollary 10.9]{RoWe98} and the smoothness of the affine function $\phi(\y) + \scal{\cdot-\y}{\nabla \phi(\y)}$, we observe that 
$$
\partial (f+\D[\phi](\cdot,\y))(\x)=\partial {\left(f+\frac1\lambda \phi\right)}(\x) - \frac1\lambda\nabla \phi(\y).
$$
Therefore, invoking Fermat's rule~\cite[Theorem 10.1]{RoWe98}, the first inclusion follows:
\begin{align*}
\x \in \lprox{\phi}{f}{\lambda}(\y) &\implies 0 \in \partial(\phi + \lambda f)(\x) - \nabla \phi(\y) \\
&\implies \x \in {\left(\partial(\phi + \lambda f)^{-1}\circ \nabla \phi \right)}{(\y)}.
\end{align*}
The second inclusion is clear. For the third inclusion note that 
$$
{\ran} {\left((\partial(\phi + \lambda f))^{-1} \right)} = \dom \partial(\phi + \lambda f).
$$
Let $\x \in \dom \partial(\phi + \lambda f)$. This means in particular $\x \in \dom f \cap \dom \phi$ and there exists $\q \in \bR^m$ such that $\q \in \partial(\phi + \lambda f)(\x)$. In view of condition~\eqref{eq:qualification_interior}, we invoke Lemma~\cite[Corollary 10.9]{RoWe98} and \cite[Proposition 8.12]{RoWe98} to obtain 
$$
\q \in \partial(\phi + \lambda f)(\x)\subset \partial\phi(\x) +\lambda \partial f(\x).
$$
This shows that the subset relation is preserved under the $\dom$-operation: $\dom\partial(\phi + \lambda f)\subset \dom \partial\phi \cap \dom \partial f$.
In addition, since $\phi$ is essentially smooth we know from Lemma~\ref{lem:legendre_props} that $\dom \partial\phi = \intr (\dom \phi)$. This yields 
$$
\dom \partial\phi \cap \dom \partial f \subset \dom \partial\phi = \intr (\dom \phi)
$$
and overall $\ran ((\partial(\phi + \lambda f))^{-1})\subset \intr (\dom \phi)$. 
\qed
\end{proof}
We remark that \eqref{eq:qualification_interior} is the standard CQ that ensures the sum-rule \cite[Corollary 10.9]{RoWe98} to hold with inclusion: $\partial (\phi+f) \subset \partial \phi+\partial f$. The condition is guaranteed to hold everywhere if for instance $f$ is smooth, cf. \cite[Exercise 8.8]{RoWe98} or $\dom \phi$ is open or simply $\dom f \subset \intr(\dom \phi)$. The conclusion also follows when $f \in \Gamma_0(\bR^m)$ and $\intr(\dom f) \cap \intr(\dom \phi) \neq \emptyset$.

We state the notion of the right Bregman--Moreau envelope and associated proximal mapping for some step-size parameter $\lambda>0$ adopting the definition from \cite{bauschke2017regularizing} for the convex setting. We would highlight a close connection between the left and the right Bregman proximal mapping, which we invoke in the course of this work to transfer results from the left to the right Bregman proximal mapping.

\begin{definition}[Right Bregman--Moreau envelope and proximal mapping] \label{def:right_env_prox}
Let $\phi \in \Gamma_0(\bR^m)$ be Legendre and $\map{f}{\bR^m}{\exR}$ be proper. For some $\lambda > 0$ and $\y \in \bR^m$ we define the right Bregman--Moreau envelope (in short: right envelope) of $f$ at $\y$ as
\begin{align} \label{eq:right_env}
\renv{\phi}{f}{\lambda}(\y)&=\inf_{\x \in \bR^m} f(\x) + \frac{1}{\lambda} \D[\phi](\y,\x),
\end{align}
and the associated right Bregman proximal mapping (in short: right prox) of $f$ at $\y$ as
\begin{align} \label{eq:right_prox}
\rprox{\phi}{f}{\lambda}(\y)=\argmin_{\x \in \bR^m} f(\x) + \frac{1}{\lambda} \D[\phi](\y,\x).
\end{align}
\end{definition}

\cite{bauschke2009bregman,bauschke2011chebyshev} studied the single-valuedness of the (nonconvex) right Bregman projection\footnote{We write ``nonconvex Bregman projection'' or ``nonconvex Bregman proximal mapping'' for the sake of convenience, whereby they mean Bregman projection with respect to a (possibly) nonconvex set or Bregman proximal mapping with respect to a (possibly) nonconvex function.}, through the left Bregman projection, expressing the right Bregman projection as a transformed left Bregman projection via the identity $\D[\phi](\y, \x) = \D[\phi^*](\nabla \phi(\x), \nabla \phi(\y))$ for $\x,\y \in \intr (\dom \phi)$, see \cite[Proposition 7.1]{bauschke2009bregman} or \cite[Lemma 2.1]{bauschke2011chebyshev}. We adapt their approach to the left and right nonconvex Bregman proximal mapping and obtain the following relation between the two.

\begin{lemma} \label{lem:right_prox}
Let $\phi \in \Gamma_0(\bR^m)$ be Legendre and let $\map{f}{\bR^m}{\exR}$ be proper such that $\intr (\dom \phi) \cap \dom f \neq \emptyset$ and let $\y \in \intr(\dom \phi)$. Then we have for the right Bregman envelope
\begin{align}
\renv{\phi}{f}{\lambda}(\y) = \lenv{\phi^*}{(f \circ \nabla \phi^*)}{\lambda}(\nabla \phi(\y)),
\end{align}
and the associated right Bregman-prox
\begin{align}
\rprox{\phi}{f}{\lambda}(\y) = {\nabla \phi^*} {\big(\lprox{\phi^*}{(f \circ \nabla \phi^*)}{\lambda} (\nabla \phi(\y)) \big)}.
\end{align}
\end{lemma}
\begin{proof}
Let $\y \in \intr(\dom \phi)$. In view of Lemma~\ref{lem:bregman_props}(ii), we may introduce a substitution $\z = \nabla \phi(\x)$, for $\x \in \intr(\dom \phi)$ and rewrite due to the properness of $f$:
\begin{align*}
\renv{\phi}{f}{\lambda}(\y) &= \inf_{\x \in \bR^m} f(\x) + \frac1\lambda \D[\phi](\y, \x) \\
&= \inf_{\x \in \intr(\dom \phi)} f(\x) + \frac1\lambda \D[\phi](\y, \x) \\
&= \inf_{\x \in \intr(\dom \phi)} f(\x) + \frac1\lambda \D[\phi^*](\nabla \phi(\x), \nabla \phi(\y)) \\
&= \inf_{\z \in \intr(\dom \phi^*)} f( \nabla \phi^*(\z)) + \frac1\lambda \D[\phi^*](\z, \nabla \phi(\y)) \\
&= \lenv{\phi^*}{(f \circ \nabla \phi^*)}{\lambda}(\nabla \phi(\y)).
\end{align*}
By the same argument, we also have:
\begin{alignat*}{2}
&\qquad& \x &\in \rprox{\phi}{f}{\lambda}(\y) \\
\iff&& \x &\in \argmin_{\x \in \intr(\dom \phi)} f(\x) + \frac{1}{\lambda} \D[\phi^*](\nabla \phi(\x), \nabla \phi(\y)) \\
\iff&& \nabla \phi(\x) &\in \argmin_{\z \in \intr(\dom \phi^*)} f(\nabla \phi^*(\z)) + \frac{1}{\lambda} \D[\phi^*](z, \nabla \phi(\y)) \\
\iff&& \x &\in {\nabla \phi^*}{\big(\lprox{\phi^*}{(f \circ \nabla \phi^*)}{\lambda}(\nabla \phi(\y)) \big)}.
\end{alignat*}
\qed
\end{proof}

The above relation between left and right envelope reveals that prox-boundedness of $f \circ \nabla \phi^*$ relative to $\phi^*$ is equivalent to $\renv{\phi}{f}{\lambda}(\y) > -\infty$ for some $\y \in \intr(\dom \phi)$ and $\lambda > 0$. This leads us to formulate an analogue definition of prox-boundedness for the right prox.
\begin{definition}[Relative right prox-boundedness]
Let $\phi \in \Gamma_0(\bR^m)$ be Legendre and $\map{f}{\bR^m}{\exR}$ be proper such that $\intr (\dom \phi) \cap \dom f \neq \emptyset$. We say $f$ is right prox-bounded relative to $\phi$ if there exists $\lambda > 0$ such that $\renv{\phi}{f}{\lambda}(\y)> -\infty$ for some $\y\in \intr(\dom \phi)$. The supremum of the set of all such $\lambda$ is the threshold $\lambda_f$ of the right prox-boundedness, i.e.
$$
\lambda_f = \sup\big\{ \lambda > 0: \exists\, \y \in \intr(\dom \phi) : \renv{\phi}{f}{\lambda}(\y)> -\infty \big\}.
$$
\end{definition}

The above expression for the right Bregman proximal mapping shows that the continuity properties of the left prox carry over to right prox under right prox-boundedness and super-coercivity of $\phi^*$, or equivalently, in view of Lemma~\ref{lem:legendre_props}(iv), $\dom \phi = \bR^m$. Full domain of $\phi$ has appeared as an assumption in several earlier related works that are concerned with the single-valuedness of the right Bregman projection \cite{bauschke2009bregman,bauschke2011chebyshev}, while \cite{bauschke2011chebyshev} posed the open question as to whether these assumptions are really necessary in the context of Chebyshev sets, see \cite[Problem 2]{bauschke2011chebyshev}.
\begin{lemma}[Continuity properties of the right prox and envelope] \label{lem:continuity_right}
Let $\phi \in \Gamma_0(\bR^m)$ be Legendre with $\dom \phi = \bR^m$ and let $\map{f}{\bR^m}{\exR}$ be proper, lsc and relatively right prox-bounded with threshold $\lambda_{f}$ and let $\lambda \in \left]0,\lambda_f\right[$. 
Then $\rprox{\phi}{f}{\lambda}$ and $\renv{\phi}{f}{\lambda}$ have the following properties:
\begin{enumerate}
\item[\rm (i)] $\rprox{\phi}{f}{\lambda}(\y) \neq \emptyset$ is compact for all $\y \in \bR^m$ and the envelope $\renv{\phi}{f}{\lambda}$ is proper.
\item[\rm (ii)] The envelope $\renv{\phi}{f}{\lambda}$ is continuous.
\item[\rm (iii)] For any sequence $\y^\nu \to \y^*$ and $\x^\nu \in \rprox{\phi}{f}{\lambda}(\y^\nu)$ we have $\{\x^\nu\}_{\nu \in \bN}$ is bounded and all its cluster points $\x^*$ lie in $\rprox{\phi}{f}{\lambda}(\y^*)$.
\end{enumerate}
\end{lemma}
\begin{proof}
Since $f$ is right prox-bounded relative to $\phi$ with threshold $\lambda_f$, due to the identity $\renv{\phi}{f}{\lambda}(\y) = \lenv{\phi^*}{(f \circ \nabla \phi^*)}{\lambda}(\nabla \phi(\y))$ for any $\y \in \bR^m$ (which holds thanks to Lemma~\ref{lem:right_prox}), we know that $f \circ \nabla \phi^*$ is (left) prox-bounded relative to $\phi^*$ with the same threshold. The result then follows by invoking Lemma~\ref{lem:continuity_left} for $f \circ \nabla \phi^*$, the super-coercivity of $\phi^*$ and the continuity of both $\nabla \phi$ and $\nabla \phi^*$, cf. Lemma~\ref{lem:legendre_props}.
\qed
\end{proof}

\subsection{Single-Valuedness of Bregman Proximal Mappings under Relative Prox-regularity} \label{sec:relative_prox_regularity}
We generalize the definition of proximal subgradients \cite[Definition 8.45]{RoWe98} to a Bregman distance setting: A classical proximal subgradient is a regular subgradient for which the error term $o(\|\x-\bar{\x}\|)$ can be specialized to a negative quadratic: $o(\|\x-\bar{\x}\|) = -(r/2)\|\x-\bar{\x}\|^2$. Analogously, a relatively proximal subgradient is a regular subgradient where the error term $o(\|\x-\bar{\x}\|)$ specializes to a Bregman distance $- r\D[\phi](\x,\bar{\x})$.
\begin{definition}[Relatively proximal subgradients and normals] Let $\phi \in \Gamma_0(\bR^m)$ be Legendre. A vector $\q$ is called a relatively proximal subgradient (relative to $\phi$) of a function $\map{f}{\bR^m}{\exR}$ at $\bar{\x} \in \intr (\dom \phi)$, a point where $f(\bar{\x})$ is finite, if there exist $r>0$ and $\epsilon >0$ such that for all $\|\x-\bar{\x}\|\leq \epsilon$ it holds that $\x \in \intr (\dom \phi)$ and
\begin{align} \label{eq:prox_subgradient}
f(\x)\geq f(\bar{\x})+\langle \q, \x-\bar{\x}\rangle - r\D[\phi](\x,\bar{\x}).
\end{align}
If $f=\delta_C$ specializes to an indicator function of a set $C$ we shall refer to $\q$ as a relatively proximal normal to $C$.
\end{definition}

We shall point out the following relation to classical proximal subgradients and normals \cite[Definition 8.45]{RoWe98}, i.e. when $\phi=(1/2)\|\cdot\|^2$.
\begin{proposition} \label{prop:equivalence_prox_subgrad}
Let $\phi \in \Gamma_0(\bR^m)$ be Legendre and $\mathcal{C}^2$ on $\intr (\dom \phi)$. Let $\map{f}{\bR^m}{\exR}$ be finite at $\bar{\x} \in \intr (\dom \phi)$. Then the following implication holds:
If $\q\in \bR^m$ is a relatively proximal subgradient of $f$ at $\bar{x}$, then $\q$ is a proximal subgradient of $f$ at $\bar{x}$. The converse is true if, furthermore, $\nabla^2\phi(\x)$ is positive definite for any $\x \in\intr (\dom \phi)$, i.e. $\phi$ is very strictly convex.
\end{proposition}
\begin{proof}
This is a direct consequence of Lemma~\ref{lem:bregman_props}(iii).
\qed
\end{proof}

\begin{figure}[h]
\centering
\begin{subfigure}[b]{0.3\linewidth}
        \centering
%        \resizebox{\textwidth}{!}{
%		\input{fig1.tex}
% 	}
        \includegraphics[width=\textwidth]{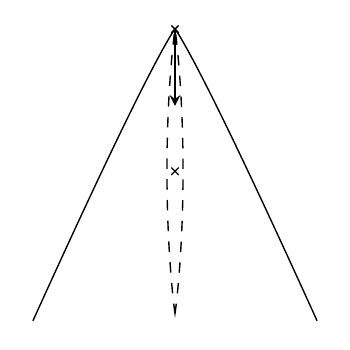}	
\end{subfigure}
\hspace{0.05\linewidth}
\begin{subfigure}[b]{0.3\linewidth}
        \centering
%        \resizebox{\textwidth}{!}{
%		\input{fig2.tex}
% 	}
	 \includegraphics[width=\textwidth]{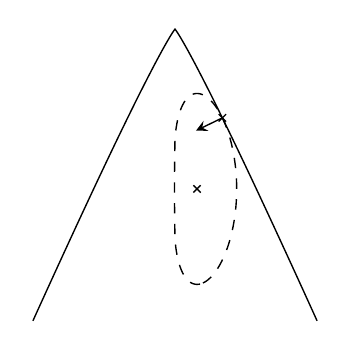} 
\end{subfigure}
\caption{Illustration of Bregman proximal normals and Bregman projections by means of Example~\ref{ex:amenable}. The set $C=\epi h$, given as the epigraph of $h(\x) = 2\x^2 - 3|\x|^{1.1}$, is indicated as the area above the solid line, which corresponds to the graph of $h$ around $0$. The arrows indicate the relatively proximal normals $\q$ of $C$ at $\bar \x$. The dashed line corresponds to the Bregman distance ball around the point $\nabla \phi^*(\nabla \phi(\bar \x) + \lambda \q)$ generated by $\phi(\x_1, \x_2)=\x_1^2 + |\x_1|^{1.1} + \x_2^2$, whose ``radius'' is chosen such that its upper surface touches the epigraph of $h$ only at $\bar \x$, which is possible if $\lambda>0$ is sufficiently small. While in the right figure, $\q$ is both a relatively proximal normal and a classical proximal normal of $C$ at $\bar \x$, in the left figure the situation is different: While it is a relatively proximal normal of $C$ at $\bar \x$, it is not a classical proximal normal, since for any $\lambda >0$ there is no Euclidean ball around $\bar \x + \lambda \q$, below $C$ that touches $C$ only at $0$. In view of the subgradient inequality~\eqref{eq:prox_subgradient}, $\tilde v:=\q + (1/\lambda)\nabla \phi(\bar \x)$ is a classical proximal subgradient of $\tilde{f}:=(1/\lambda)\phi + \delta_C$ at $\bar \x$ though, cf. also Remark~\ref{rem:strong_convexity}.}
\label{fig:bergman_proj}
\end{figure}

\begin{lemma}[Globalization of proximal subgradient inequality] \label{lem:global_ineq}
Let $\phi \in \Gamma_0(\bR^m)$ be Legendre and super-coercive and $\map{f}{\bR^m}{\exR}$ be proper lsc, relatively prox-bounded with threshold $\lambda_f$ and finite at $\bar{\x}\in \intr (\dom \phi)$. Let $\bar{\q}$ be a relatively proximal subgradient of $f$ at $\bar{\x}$. Then, if $r>0$ is sufficiently large the subgradient inequality \eqref{eq:prox_subgradient} holds globally for all $\x \in \bR^m$.
\end{lemma}
\begin{proof}
Since $\bar{\q}$ is a relatively proximal subgradient of $f$ at $\bar{\x}$ we know that there exists $r' > 0$ and $\epsilon > 0$ such that for any $r\geq r'$ we have
\begin{equation} \label{eq:lem:global_ineq:A}
f(\x)\geq f(\bar{\x})+\langle \bar{\q}, \x-\bar{\x}\rangle - r \D[\phi](\x,\bar{\x}),
\end{equation}
whenever $\|\x-\bar{\x}\| < \epsilon$ and $\epsilon$ is sufficiently small such that $\x \in \intr (\dom \phi)$.
We prove the assertion by showing that the inequality also holds for any $\x\in\bR^m$ with $\|\x-\bar{\x}\|\geq \epsilon$, when $r$ is chosen to be sufficiently large:
Let $\lambda \in \left]0,\lambda_f\right[$. Since $\map{f}{\bR^m}{\exR}$ is prox-bounded and proper, lsc and $\bar{\x}\in \intr(\dom \phi)$ we know from Lemma~\ref{lem:continuity_left} that $+\infty > \lenv{\phi}{f}{\lambda}(\bar{\x})>-\infty$ since $f$ is prox-bounded and $f(\bar{\x})$ is finite.
Then we have
\begin{equation} \label{eq:lem:global_ineq:B}
f(\x) \geq \lenv{\phi}{f}{\lambda}(\bar{\x}) - \frac{1}{\lambda}\D[\phi](\x,\bar{\x}),
\end{equation}
for all $\x \in \bR^m$.
Combining \eqref{eq:lem:global_ineq:A} and \eqref{eq:lem:global_ineq:B} shows that \eqref{eq:prox_subgradient} holds with modulus $r\geq\max\set{r',1/\lambda}$, when 
$$
  \lenv{\phi}{f}{\lambda}(\bar{\x}) - \frac{1}{\lambda}\D[\phi](\x,\bar{\x}) \geq f(\bar{\x})+\langle \bar{\q}, \x-\bar{\x}\rangle - r\D[\phi](\x,\bar{\x}) 
$$
is satisfied, which is implied (using the Cauchy--Schwarz inequality) by
\begin{align} \label{eq:lem_global_bound}
\frac{f(\bar{\x})-\lenv{\phi}{f}{\lambda}(\bar{\x})}{\|\x-\bar{\x}\|} +\|\bar{\q}\| \leq \left( r - \frac{1}{\lambda}\right) \frac{\D[\phi](\x,\bar{\x})}{\|\x-\bar{\x}\|}.
\end{align}
Using super-coercivity of $\phi$, the inequality happens to be true for \mynewline $r\geq\max\set{r',1/\lambda}$ and all $\x$ with $\norm{\x-\bar{\x}}\geq \mu$ for some $\mu> \eps$. It remains to verify \eqref{eq:lem_global_bound} for $\mu > \|\x - \bar{\x}\| \geq \epsilon>0$ for some $r$. However, since for such $\x$, using strict convexity of $\phi$, obviously, $\D[\phi](\x,\bar{\x})$ is bounded away from $0$, we can find some $r$ sufficiently large such that \eqref{eq:lem_global_bound} also holds for $\x$ with $\mu > \|\x - \bar{\x}\| \geq \epsilon$.
\qed
\end{proof}

The following lemma shows that analogous to the classical prox, the left Bregman prox and envelope of a tilted function $f - \langle \cdot, \q \rangle$ can be written as the Bregman prox and envelope of $f$ at a transformed point, respectively.
\begin{lemma}[Effects of tilt transformation] \label{lem:effects_tilt}
Let $\phi \in \Gamma_0(\bR^m)$ be Legendre and $f:\bR^m\to\exR$ be proper lsc. Let $\y \in \intr (\dom \phi)$ and $ \q \in \bR^m$. Denote by $\z := \nabla \phi^*(\nabla \phi(\y) + \lambda\q)$ and by $f_0:=f - \langle \cdot, \q \rangle$.
Then we have the following identities for the prox
$$
\lprox{\phi}{f_0}{\lambda}(\y) = \lprox{\phi}{f}{\lambda}(\z),
$$
and the envelope function
\begin{align*}
\lenv{\phi}{f_0}{\lambda}(\y) = \lenv{\phi}{f}{\lambda}(\z) +\frac{1}{\lambda}\D[\phi](\z, \y) -\langle  \q, \z \rangle.
\end{align*}
\end{lemma}
\begin{proof}
For $\z = \nabla \phi^*(\nabla \phi(\y) + \lambda\q)$ the identities follow from the following calculation:
\begin{align*}
\D[\phi](\x, \z) &=  \phi(\x) - \phi(\z) - \langle \nabla \phi(\y) + \lambda\q, \x - \z \rangle \\
&= -\lambda\langle \q, \x \rangle + \D[\phi](\x,\y) +  \phi(\y) - \langle \nabla \phi(\y), \y \rangle +\langle \nabla \phi(\y) + \lambda\q, \z \rangle- \phi(\z).
\end{align*}
Scaling the equality with $1/\lambda$ and adding $f(\x)$ and reordering yields
\begin{align*}
f_0(\x) + \frac{1}{\lambda}\D[\phi](\x,\y) &=  f(\x) + \frac{1}{\lambda} \big(\D[\phi](\x, \z) - \phi(\y) \\
&\qquad+ \langle \nabla \phi(\y), \y \rangle - \langle \nabla \phi(\y) + \lambda\q, \z\rangle +  \phi(\z) \big)\\
&= f(\x) + \frac{1}{\lambda} \D[\phi](\x, \z) + \frac{1}{\lambda}\D[\phi](\z, \y) -\langle  \q, \z \rangle.
\end{align*}
 \qed
\end{proof}

Based on the globalized subgradient inequality from Lemma~\ref{lem:global_ineq} we shall characterize relatively proximal subgradients via the Bregman proximal map. This property is used frequently in the course of this section to assert the single-valuedness of the Bregman proximal mapping.
\begin{proposition} \label{prop:subgradient}
Let $\phi \in \Gamma_0(\bR^m)$ be Legendre and $\map{f}{\bR^m}{\exR}$ be proper lsc and finite at $\bar{\x}\in \intr (\dom \phi)$. Then the following conditions are equivalent for some $\lambda > 0$:
\begin{enumerate}
\item[\rm (i)] The following inclusion holds:
\begin{align} \label{eq:cor_subgradient_incl}
\bar \x \in \lprox{\phi}{f}{\lambda}(\nabla \phi^*(\nabla \phi(\bar x) + \lambda \q)).
\end{align}
\item[\rm (ii)] The subgradient inequality~\eqref{eq:prox_subgradient} holds globally for all $\x \in \bR^m$ and $r:=1/\lambda$,
\end{enumerate}
where strict inequality holds for all $\x \neq \bar{\x}$ by decreasing $\lambda$. Equivalently this means the inclusion~\eqref{eq:cor_subgradient_incl} holds with equality.
If, furthermore, $\phi$ is super-coercive and $f$ relatively prox-bounded with threshold $\lambda_f$ then the above conditions hold for some $\lambda< \lambda_f$ sufficiently small, if and only if $\q \in \bR^m$ is a relatively proximal subgradient of $f$ at $\bar{x}$.
\end{proposition}
\begin{proof}
Let $\lambda > 0$ and let the subgradient inequality~\eqref{eq:prox_subgradient} hold globally for all $\x \in \bR^m$ and $r:=1/\lambda$. This means:
\begin{align} 
f(\x)-\langle \q, \x\rangle \geq f(\bar{\x})-\langle \q, \bar{\x}\rangle - \frac{1}{\lambda}\D[\phi](\x,\bar{\x}).
\end{align}
Define $f_0:= f - \langle\q, \cdot \rangle$. Then, by reordering the terms the above is equivalent to:
\begin{align} \label{eq:proximal_subgradient:A}
f_0(\x) + \frac{1}{\lambda}\D[\phi](\x,\bar{\x})\geq f_0(\bar{\x}) + \frac{1}{\lambda}\D[\phi](\bar \x,\bar{\x}),
\end{align}
which holds if and only if $\bar{\x} \in \lprox{\phi}{f_0}{\lambda}(\bar \x)$, which, in view of Lemma~\ref{lem:effects_tilt}, is equivalent to $\bar \x \in \lprox{\phi}{f}{\lambda}( \nabla \phi^*(\nabla \phi(\bar \x) + \lambda\q))$.
Then, clearly, strict inequality holds for all $\x \neq \bar{\x}$ by decreasing $\lambda$, which equivalently means the inclusion~\eqref{eq:cor_subgradient_incl} holds with equality.

Let $\q$ be a relatively proximal subgradient of $f$ at $\bar x$ with constants $\epsilon>0$ and $r>0$. Then we may invoke Lemma~\ref{lem:global_ineq} to make the subgradient inequality in \eqref{eq:prox_subgradient} hold globally, for all $\x\in \bR^m$, $r:=1/\lambda$ and $\lambda>0$ sufficiently small. Conversely, when the subgradient inequality~\eqref{eq:prox_subgradient} holds globally for $r:=1/\lambda$, this means in particular that $\q$ is a relatively proximal subgradient.
%This inclusion holds with equality (strict inequality in \eqref{eq:proximal_subgradient:B}) whenever $\lambda \in (0, \lambda')$.
\qed
\end{proof}

An important class of prox-bounded functions $f$ are indicator functions $f=\delta_C$ \xspace of a possibly nonconvex closed set $C$. Indeed, the threshold of prox-boundedness for such $f$ is $\lambda_f = \infty$. Invoking the above lemma we obtain that $\q$ is a relatively proximal normal to $C$ at $\bar{x}$ if and only if we can perturb $\bar x$ along $\q$ in the Bregman sense as $\y:= \nabla \phi^*(\nabla \phi(\bar x) + \lambda \q)$ so that by Bregman-projecting the perturbed point $\y$ back on $C$ (i.e. computing the left prox of $f$ at $\y$), we recover $\bar \x$. Indeed for $\phi = (1/2)\|\cdot\|^2$ we obtain the classical definition of proximal normals:
$$
N_C^P(\bar x) := \left\{ r(\y - \bar x) : \bar \x \in \mathrm{proj}_C(\y), r \geq 0, \y \in \bR^m \right\},
$$
where $\mathrm{proj}_C$ denotes the classical Euclidean projection onto the set $C$. This is illustrated in Figure~\ref{fig:bergman_proj}.

We now define relative prox-regularity, generalizing \cite[Definition 13.27]{RoWe98} to a Bregman distance setting: We fix a reference point $(\bar{\x}, \bar{\q})$, where $f(\bar{\x})$ is finite and $\bar{\q} \in \partial f(\bar{\x})$ is a relatively proximal subgradient, and require the subgradient inequality \eqref{eq:prox_subgradient} to hold uniformly on a $f$-attentive neighborhood of $(\bar{\x}, \bar{\q})$:

\begin{definition}[Relative prox-regularity]
Let $\phi \in \Gamma_0(\bR^m)$ be Legendre. We say a function $\map{f}{\bR^m}{\exR}$ is relatively prox-regular at $\bar{\x}\in \intr (\dom \phi)$ for $\bar{\q}\in \bR^m$ if $f$ is finite and locally lsc at $\bar \x$ with $\bar{\q}\in\partial f(\bar{\x})$ and there exist $\epsilon>0$ and $r \geq 0$ such that for all $\|\x'-\bar{\x}\|<\epsilon$, $\|\x-\bar{\x}\|<\epsilon$, with $\epsilon$ sufficiently small such that $\x,\x' \in \intr(\dom \phi)$, it holds that:
\begin{align}  \label{eq:prox_regularity}
f(\x') \geq f(\x)+\iprod{\q}{\x'-\x}-r \D[\phi](\x',\x), 
\end{align}
whenever $f(\x)-f(\bar{\x})<\epsilon$, $\q\in\partial f(\x)$, $\|\q-\bar{\q}\|<\epsilon$. When this property holds for all $\bar{\q} \in \partial f(\bar{\x})$, $f$ is said to be relatively prox-regular at $\bar{\x}$.
\end{definition}

%The condition in the definition holds uniformly for some $r>0$ at all $\bar{\x}\in \intr (\dom \phi) \cap \dom f$ with $\epsilon=\infty$ and $\x,\x' \in \intr(\dom \phi)$ is equivalent to relative hypoconvexity, i.e. the convexity of $f + r \phi$ on $\intr(\dom \phi)$. In view of the chain-rule $f + r \phi$ is convex on $\intr(\dom \phi)$ for proper lsc $f$ if and only if $\dom f \cap \intr(\dom \phi)$ is convex and:
%\begin{align}
%f(\x')+r\phi(\x') \geq f(\x)+\iprod{\q + r\nabla \phi(\x)}{\x'-\x}+r\phi(\x).
%\end{align}

For examples of relatively prox-regular functions we refer to Section~\ref{sec:examples_prox_regular} below. We first clarify a relation between the new relative prox-regularity property and classical prox-regularity.
\begin{proposition} \label{prop:prox_regular_very_strictly_convex}
Let $\phi \in \Gamma_0(\bR^m)$ be Legendre and $\mathcal{C}^2$ on $\intr (\dom \phi)$. Let $\map{f}{\bR^m}{\exR}$ be extended real-valued and $\bar{\x} \in \intr (\dom \phi)$. Then the following implication holds:
If $f$ is relatively prox-regular at $\bar{\x}$ for $\bar{\q}$, then $f$ is also prox-regular at $\bar{\x}$ for $\bar{\q}$. The converse is true if, furthermore, $\nabla^2\phi(\x)$ is positive definite for any $\x \in\intr (\dom \phi)$, i.e. $\phi$ is very strictly convex.
\end{proposition}
\begin{proof}
This is a direct consequence of Lemma~\ref{lem:bregman_props}(iii).
\qed
\end{proof}

The relative prox-regularity property of a tilted function is preserved as the following lemma shows:
\begin{lemma}[Invariance under tilt transformation] \label{lem:invariance_tilt}
Let $\phi \in \Gamma_0(\bR^m)$ be Legendre, $f:\bR^m\to\exR$ be extended real-valued and $\bar{\x} \in \intr (\dom \phi)$. Then $f$ is relatively prox-regular at $\bar{\x}$ for $\bar{\q}\in\partial f(\bar{\x})$ if and only if $f_0:= f - \langle\q, \cdot \rangle$ is relatively prox-regular at $\bar{\x}$ for $0\in \partial f_0(\bar{\x})$.
\end{lemma}
\begin{proof}
This is clear from the definition of relative prox-regularity.
\qed
\end{proof}

The following theorem is analogous to \cite[Theorem 3.2]{PoRo96} or \cite[Theorem 13.36]{RoWe98} for classical prox-regularity. More precisely, for a function $f$ and a reference point $(\bar \x, \bar \q) \in \gph \partial f$ we provide an equivalent characterization of relative prox-regularity in terms of the relative hypomonotonicity of an $f$-attentive graphical localization $T$ of the subdifferential $\partial f$ at $(\bar \x, \bar\q)$, defined as follows:
\begin{definition}[$f$-attentive localization of limiting subdifferential]
Let $\map{f}{\bR^m}{\exR}$ be finite at $\bar \x$. Let $\bar \q \in \partial f(\bar \x)$. Then for some $\eps > 0$ the $f$-attentive $\epsilon$-localization $T:\bR^m \rightrightarrows \bR^m$ of $\partial f$ around $(\bar \x, \bar \q)$ is defined by
\begin{align}
T(\x):=\begin{cases} \{\q \in \partial f(\x) : \|\q -\bar{\q}\|< \epsilon\} & \text{if $\|\x - \bar{\x}\| < \epsilon$ and $f(\x)<f(\bar{\x})+\epsilon$},\\
\emptyset, & \text{otherwise.}
\end{cases}
\end{align}
\end{definition}
Such a statement can be seen as a localized analogue to the equivalence between the relative hypoconvexity of $f$, i.e. $f+r\phi$ is convex on $\intr(\dom \phi)$ for some $r\geq0$, and the relative hypomonotonicity of $\partial f$.
An important difference to note is that, in the following statement, we also require $\bar \q$ to be a relatively proximal subgradient of $f$ at $\bar \x$.
\begin{theorem} \label{thm:prox_regularity_equiv}
Let $\phi \in \Gamma_0(\bR^m)$ be Legendre and super-coercive and the function $\map{f}{\bR^m}{\exR}$ be proper lsc, prox-bounded with threshold $\lambda_f$ and finite at $\bar{\x}\in \intr (\dom \phi)$. Then the following conditions are equivalent:
\begin{enumerate}
\item[\rm (i)] $f$ is relatively prox-regular at $\bar{\x}$ for $\bar{\q}$.
\item[\rm (ii)] $\bar{\q}\in\partial f(\bar{\x})$ is a relatively proximal subgradient and $\partial f$ has an $f$-attentive $\epsilon$-localization $T:\bR^m \rightrightarrows \bR^m$ around $(\bar{\x},\bar{\q})$ such that $T +r \nabla \phi$ is monotone for some $r>0$.

\item[\rm (iii)] For $\bar{\q} \in \partial f(\bar{\x})$ and $\lambda < \lambda_f$ sufficiently small it holds that $\lprox{\phi}{f}{\lambda}$ is a singled-valued map near the point $\bar{\y}:=\nabla \phi^*(\nabla\phi(\bar{\x}) + \lambda \bar{\q})$ such that \mynewline $\{\bar{\x} \}=\lprox{\phi}{f}{\lambda}(\bar{\y})$ and
\begin{align} \label{eq:equality_prox}
\lprox{\phi}{f}{\lambda}(\y) = {\left((\nabla \phi + \lambda T)^{-1} \circ \nabla \phi \right)}(\y),
\end{align}
for some $f$-attentive $\epsilon$-localization $T:\bR^m \rightrightarrows \bR^m$ of $\partial f$ around $(\bar{\x},\bar{\q})$ and $\y$ near $\bar{\y}$.
\end{enumerate}
\end{theorem}
\begin{proof} 
(i) $\implies$ (ii): Let $f$ be relatively prox-regular at $\bar{\x}$ for $\bar{\q}\in\partial f(\bar{\x})$. This means that there exist constants $\epsilon >0$ and $r>0$ such that the subgradient inequality \eqref{eq:prox_regularity} holds for $\x',\x \in \bR^m$ with $\|\x'-\bar{\x}\| < \epsilon$, $\|\x-\bar{\x}\| < \epsilon$ and $\q\in \partial f(\x)$, $\q' \in \partial f(\x')$, $\|\q'-\bar{\q}\| < \epsilon$, $\|\q-\bar{\q}\| < \epsilon$. In particular this implies that $\bar{\q}$ is a relatively proximal subgradient at $\bar{\x}$ and we have:
$$
f(\x') \geq f(\x)+\iprod{\q}{\x'-\x}-r \D[\phi](\x',\x), \quad f(\x) \geq f(\x')+\iprod{\q'}{\x-\x'}-r \D[\phi](\x,\x').
$$
Adding these inequalities yields:
\begin{align*}
0 &\geq \iprod{\q}{\x'-\x} + \iprod{\q'}{\x-\x'}-r\big(\D[\phi](\x',\x) + \D[\phi](\x,\x')\big) \\
&=-\iprod{\q-\q'}{\x-\x'} - r\iprod{\nabla \phi(\x) - \nabla \phi(\x')}{\x-\x'}.
\end{align*}
This shows that the corresponding map $T+r\nabla \phi$ is monotone, where $T$ is the $f$-attentive $\epsilon$-localization of $\partial f$ at $(\bar{\x},\bar{\q})$.

(ii) $\implies$ (iii): By assumption, $\bar{\q}$ is a relatively proximal subgradient. Then we may invoke Proposition~\ref{prop:subgradient} to obtain that $\{\bar{\x}\} = \lprox{\phi}{f}{\lambda}(\bar{\y})$ is a singleton for $\lambda < \min\{\lambda_f, 1/r\}$ being sufficiently small. Due to the prox-boundedness, we can invoke Lemma~\ref{lem:continuity_left} to assert that $\lprox{\phi}{f}{\lambda}(\y) \neq \emptyset$ for any $\y\in \intr (\dom \phi)$. Furthermore, for any sequence $\x^\nu \in \lprox{\phi}{f}{\lambda}(\y^\nu)$, $\y^\nu \to \bar{\y}$ we have $\{\x^\nu\}_{\nu \in \bN}$ is bounded and all its cluster points lie in $\lprox{\phi}{f}{\lambda}(\bar{\y}) = \{\bar{\x}\}$, meaning $\x^\nu \to \bar{\x}$ and $\lenv{\phi}{f}{\lambda}(\y^\nu) \to \lenv{\phi}{f}{\lambda} (\bar{\y})$. In addition we have $f(\x^\nu) \to f(\bar{\x})$ as 
$$
\lenv{\phi}{f}{\lambda}(\y^\nu) = f(\x^\nu) + \frac1\lambda\D[\phi](\x^\nu,\y^\nu) \to \lenv{\phi}{f}{\lambda}(\bar{\y})=f(\bar{\x}) +\frac1\lambda\D[\phi](\bar{\x},\bar{\y}).
$$
Overall this shows that for any $\y$, which is sufficiently near to $\bar{\y}$, we have $\x\in \lprox{\phi}{f}{\lambda}(\y)$, $\|\x-\bar{\x}\| < \epsilon$, $|f(\x) - f(\bar{\x})| < \epsilon$ and $\|\q -\bar{\q}\| <\epsilon$, for \mynewline $\q:=(1/\lambda) \nabla\phi(\y)-(1/\lambda) \nabla\phi(\x)$ due to the continuity of $\nabla \phi$ (cf. Lemma~\ref{lem:legendre_props}).
By applying Fermat's rule~\cite[Theorem 10.1]{RoWe98} to $\lprox{\phi}{f}{\lambda}(\y)$ we obtain
$$
0 \in \partial f(\x) +\frac{1}{\lambda}(\nabla \phi(\x)-\nabla \phi(\y)),
$$
or equivalently
$$
0 \in T(\x) + r (\nabla \phi(\x)-\nabla \phi(\y)),
$$
where $\partial f(\x)$ is replaced by $T(\x)$ due to the arguments above.
This means 
\begin{align*}
\emptyset &\neq \lprox{\phi}{f}{\lambda}(\y) \subset \left\{\x \in \bR^m: 0 \in T(\x) + \frac{1}{\lambda} (\nabla \phi(\x)-\nabla \phi(\y)) \right\} \\
&=\big((\nabla \phi + \lambda T)^{-1} \circ \nabla \phi \big)(\y),
\end{align*}
which is at most a singleton due to the strict monotonicity of $T+(1/\lambda)\nabla \phi$ for $\lambda < 1/r$. This implies $\{\x\}=\lprox{\phi}{f}{\lambda}(\y)$ is a singleton for $\y$ near $\bar{\y}$.

(iii) $\implies$ (i):
Let $T$ be some $f$-attentive $\epsilon$-localization of $\partial f$ at $\bar \x$ for $\bar \q$, which has the properties in (iii). Let $\x\in\bR^m$ with $\|\x-\bar{\x}\| < \epsilon$, $f(\x) < f(\bar{\x})+\epsilon$ and $\q \in \partial f(\x)$, $\|\q-\bar{\q}\| <\epsilon$. We have $\q \in T(\x)$ and for $\epsilon >0$ sufficiently small $\x\in\intr(\dom \phi)$ and $\y:=\nabla \phi^*(\nabla \phi(\x) +\lambda \q)$ near $\nabla \phi^*(\nabla \phi(\bar{\x}) +\lambda \bar{\q}))$, due to the continuity of $\nabla \phi^*$ guaranteed by Lemma~\ref{lem:legendre_props}.
Then for such $\y$ we have that $\x \in ((\nabla \phi + \lambda T)^{-1} \circ \nabla \phi)(\y)$ and by assumption 
$$
\lprox{\phi}{f}{\lambda}(\y)=\big((\nabla \phi + \lambda T)^{-1} \circ \nabla \phi \big)(\y) \ni \x.
$$
Invoking Proposition~\ref{prop:subgradient} we obtain the subgradient inequality \eqref{eq:prox_regularity} for $r:=1/\lambda$, which holds even globally, cf. Lemma~\ref{lem:global_ineq}. We may conclude $f$ is relatively prox-regular at $\bar{\x}$ for $\bar{\q}$.
\qed
\end{proof}

\begin{remark}
We would highlight that items (i) and (ii) in the above theorem only depend on the local structure of the epigraph of $f$ near $(\bar \x,f(\bar \x))$, while in contrast, (iii) depends on its global structure. This means that (i) resp. (ii) hold for $f$ if and only if they hold for $\tilde{f}:=f+\delta_C$ for \mynewline $C:=\{\x \in \bR^m : \|\x - \bar{\x}\| \leq \epsilon, f(\x)\leq f(\bar{\x})+\epsilon \}$
compact, where $\tilde{f}$ is always proper lsc and prox-bounded whenever $f$ is locally lsc at $\bar \x$, a point where $f$ is finite. This shows that the equivalence between (i) and (ii) holds even if we relax the globally lsc assumption towards locally lsc at $\bar \x$ and entirely drop the prox-boundedness assumption. In that sense (iii) can be seen as an auxiliary statement applied to $\tilde{f}$ to show the direction (ii) implies (i), which is also used as a strategy in the proof of \cite[Theorem 13.36]{RoWe98}.
\end{remark}
\begin{remark} \label{rem:strong_convexity}
When $\phi$ is strongly convex on compact convex subsets \mynewline $K\subset \intr(\dom \phi)$ (which is implied by very strict convexity, see Lemma~\ref{lem:bregman_props}(iii), but holds more generally, for, e.g. $\phi(\x)=(1/p)|\x|^p$, $p\in ]1,2[$), the direction (ii) implies (i) follows alternatively from \cite[Theorem 3.2]{PoRo96} or \cite[Theorem 13.36]{RoWe98}. To this end, let $\bar \q$ be a relatively proximal subgradient of $f$ at $\bar \x \in \intr(\dom \phi)$ and let $T$ be a relatively hypomonotone, $f$-attentive $\epsilon$-localization of $\partial f$ at $\bar \x$ for $\bar \q$. This means that there is $r>0$ such that $\tilde \q := \bar \q + r \nabla \phi(\bar \x)$ is a classical proximal subgradient of $\tilde{f} := f + r\phi$ at $\bar \x$ and $\widetilde T := T + r\nabla \phi$ is monotone. Furthermore $\widetilde{T}$ is a $\tilde f$-attentive graphical localization of $\partial \tilde f$ at $(\bar \x, \tilde v)$. Invoking \cite[Theorem 13.36]{RoWe98} this means that $\tilde{f}$ is classically prox-regular at $\bar \x$ for $\tilde \q$. Due to the strong convexity of $\phi$ on compact convex subsets $K\subset \intr(\dom \phi)$ we can bound the negative quadratic term $-(1/2)\|\x'-\x\|^2$ in the classical subgradient inequality (locally) by a Bregman distance $-\theta\D[\phi](\x',\x)$. Rewriting the estimate gives us the result. In the general case, however, existing theory cannot be applied in its present form, which leads us to provide a generalization by means of the above theorem.
\end{remark}

\begin{corollary} \label{cor:prox_lipschitz}
Let $\phi \in \Gamma_0(\bR^m)$ be Legendre and super-coercive. Let the function $\map{f}{\bR^m}{\exR}$ be proper lsc, prox-bounded with threshold $\lambda_f$, and relatively prox-regular at $\bar{\x}\in \intr (\dom \phi)$ for $\bar{\q}$. Assume that $\phi$ is very strictly convex. Then for $\lambda < \lambda_f$ sufficiently small, $\lprox{\phi}{f}{\lambda}$ is a Lipschitz map on a neighborhood of $\bar{\y}:=\nabla\phi^*(\nabla\phi(\bar{\x}) + \lambda \bar{\q})$.
\end{corollary}
\begin{proof}
Since $f$ is relatively prox-regular at $\bar{\x}$ for $\bar{\q}\in\partial f(\bar{\x})$ due to Theorem~\ref{thm:prox_regularity_equiv} there exists $r>0$ such that $T+ r \nabla \phi$ is monotone. This means for $(\x', \q'), (\x,\q) \in \gph T$ we have:
\begin{align*}
\iprod{\q'-\q}{\x'-\x}+r\iprod{\nabla \phi(x')-\nabla \phi(x)}{\x'-\x} \geq 0.
\end{align*}
Let $\x \in \lprox{\phi}{f}{\lambda}(\y)$ and $\x' \in \lprox{\phi}{f}{\lambda}(\y')$. Due to Theorem~\ref{thm:prox_regularity_equiv} we know that $(1/\lambda)(\nabla \phi(\y) - \nabla \phi(\x)) \in T(\x)$ and $(1/\lambda)(\nabla \phi(\y') - \nabla \phi(\x')) \in T(\x')$. This means we have
\begin{align*}
\frac{1}{\lambda}\iprod{\nabla \phi(\y') - \nabla \phi(\y)}{\x'-\x} \geq \left(\frac{1}{\lambda} - r\right)\iprod{\nabla \phi(\x')-\nabla \phi(\x)}{\x'-\x} .
\end{align*}
Since $\phi$ is very strictly convex we may invoke Lemma~\ref{lem:bregman_props}(iii) to assert that there are constants $\Theta$ and $\theta$ such that for any $\x,\x' \in \intr(\dom \phi)$ near $\bar{\x}$:
\begin{align*}
 \langle \nabla \phi(\x') - \nabla \phi(\x), \x' - \x \rangle &\leq \Theta \|\x - \x'\|^2, \\
 \langle \nabla \phi(\x') - \nabla \phi(\x), \x' - \x \rangle &\geq \theta \|\x - \x'\|^2.
\end{align*}
This yields
\begin{align*}
\iprod{\nabla \phi(\y') - \nabla \phi(\y)}{\x'-\x} \geq \left(1 - \lambda r\right) \theta\|\x - \x'\|^2,
\end{align*}
and via Cauchy--Schwarz
\begin{align*}
\iprod{\nabla \phi(\y') - \nabla \phi(\y)}{\x'-\x} &\leq \|\nabla \phi(\y') - \nabla \phi(\y)\|\cdot\|\x'-\x\| \\
&\leq \Theta\|\y'-\y\|\cdot\|\x'-\x\|,
\end{align*}
and overall
\begin{align*}
\|\x - \x'\| \leq \frac{\Theta}{\theta(1 - \lambda r)} \|\y-\y'\|.
\end{align*}
\qed
\end{proof}

\subsection{Relatively Amenable Functions} \label{sec:examples_prox_regular}
An important source for examples of prox-regular functions is strong amenability \cite[Definition 10.23]{RoWe98}, i.e. functions $f$ that can locally be represented as a composition of a convex function with a smooth function and a certain constraint qualification. In the following we generalize this concept to the Bregman distance case. To this end, the recently introduced generalization of $L$-smooth functions to $L$-smooth adaptable ($L$-smad) functions \cite{BBT2016,BSTV2018} (called relatively smooth functions in \cite{LFN18}) is used. We state a slightly modified version, where we introduce an additional open subset $V \subseteq \intr(\dom \phi)$ of $\intr(\dom \phi)$ and require the property to hold only on $V$ instead of $\intr(\dom \phi)$.
\begin{definition}[Smooth adaptable function ($L$-smad)]
Let $\phi \in \Gamma_0(\bR^m)$ be Legendre. A function $\map{f}{\bR^m}{\exR}$ that is $\mathcal{C}^1$ on an open subset \mynewline $V \subseteq \intr(\dom \phi)$ is called $L$-smooth adaptable relative to $\phi$ on $V$, if there exists $L\geq0$ such that both $L\phi - f$ and $L\phi + f$ are convex on $V$.
\end{definition}
The following lemma, which we adopted from Lemma~\cite[Lemma 2.1]{BSTV2018}, is a generalization of the classical full descent lemma to the $L$-smad case:
\begin{lemma}[Full Extended Descent Lemma] \label{lem:full_descent}
Let $\phi \in \Gamma_0(\bR^m)$ be Legendre. Then, a function $\map{f}{\bR^m}{\exR}$ that is $\mathcal{C}^1$ on an open subset $V \subseteq \intr(\dom \phi)$ is $L$-smooth adaptable relative to $\phi$ on $V$ with $L\geq 0$ if and only if the following holds for all $\x,\y \in V$
$$
|f(\x) - f(\y) - \langle \nabla f(\y), \x -\y \rangle| \leq L \D[\phi](\x,\y).
$$
\end{lemma}
\begin{definition}
A function $\map{F}{\bR^m}{\bR^n}$ that is $\mathcal{C}^1$ on an open subset \mynewline$V\subseteq \intr(\dom\phi)$ is called $L$-smooth adaptable relative to $\phi$ on $V$ with $L\geq 0$ if each coordinate function $F_i$ is $L$-smooth adaptable relative to $\phi$ on $V$.
\end{definition}

We extend \cite[Definition 10.23]{RoWe98} to a setting where the inner smooth map is $L$-smooth adaptable. Note that this property is required to hold only on a local neighborhood of a reference point. The first part recapitulates the definition of an amenable function from \cite[Definition 10.23(a)]{RoWe98}, while a relatively amenable function generalizes the notion of strong amenability \cite[Definition 10.23(b)]{RoWe98}.
\begin{definition}[Relatively amenable functions]
A function $\map{f}{\bR^m}{\exR}$ is amenable at $\bar\x$, a point where $f(\bar \x)$ is finite, if there is an open neighborhood $V \subset \bR^m$ of $\bar \x$ on which $f$ can be represented in the form $f=g\circ F$ for a $\mathcal{C}^1$ mapping $\map{F}{V}{\bR^n}$ and a proper, lsc, convex function $\map{g}{\bR^n}{\exR}$ such that, in terms of $D=\cl(\dom g)$,
\begin{align} \label{eq:contraint_qualification_chain_rule}
\text{the only $\y \in N_{D}(F(\bar \x))$ with $\nabla F(\bar \x)^* y = 0$ is $\y = 0$.}
\end{align}
If the mapping $F$ is $L$-smooth adaptable relative to $\phi\in \Gamma_0(\bR^m)$ on \mynewline$V\subseteq \intr(\dom \phi)$ it is called relatively amenable at $\bar \x \in V$ relative to $\phi$.
\end{definition}
Clearly, the constraint qualification~\eqref{eq:contraint_qualification_chain_rule} is satisfied whenever $F(\bar \x) \in \intr(\dom g)$.

In the following proposition, we show that relatively amenable functions are indeed relatively prox-regular, which is completely analogous to the classical setting of strong amenability and prox-regularity \cite[Proposition 13.32]{RoWe98}. Actually, this also generalizes the classical Euclidean setting with $\phi=(1/2)\|\cdot\|^2$ to requiring the inner functions to be only $\mathcal{C}^1$ with a locally Lipschitz continuous gradient instead of $\mathcal{C}^2$.
\begin{proposition} \label{prop:amenable_implies_pr}
Let $\phi \in \Gamma_0(\bR^m)$ be Legendre and $\map{f}{\bR^m}{\exR}$ be relatively amenable at $\bar \x \in \intr(\dom \phi)$ relative to $\phi$. Then $f$ is prox-regular relative to $\phi$ at $\bar \x$.
\end{proposition}
\begin{proof}
Since $f$ is relatively amenable at $\bar \x\in\intr(\dom\phi)$ relative to $\phi$, there exists an open neighborhood $V\subseteq\intr(\dom\phi)$ of $\bar \x$ on which $f=g\circ F$ for a proper lsc convex function $\map{g}{\bR^n}{\exR}$ and a $\mathcal{C}^1$ map $\map{F}{V}{\bR^m}$ that is $L$-smooth adaptable relative to $\phi$ on $V$. Clearly, $f$ is lsc relative to $V$ and therefore in particular locally lsc at $\bar \x$. Note that the constraint qualification \eqref{eq:contraint_qualification_chain_rule} holds not only at $\bar \x$ but also on $V$, by possibly narrowing $V$. Otherwise there exists a sequence $\x^\nu \to \bar \x$ and $0 \neq \y^\nu \in N_{D}(F(\x^\nu))$ with $\nabla F(\x^\nu)^* y^\nu = 0$ where we may assume $\|\y^\nu\| = 1$ by normalizing. Taking a convergent subsequence of $\{\y^{\nu}\}_{\nu \in \bN}$ we have at the limit point $\y$ that $\nabla F(\bar \x)^* \y = 0$ and $\|\y\|=1$, which is a contradiction.
 
In view of the chain rule from~\cite[Theorem 10.6]{RoWe98}, we have for all $\x \in V$ that $\partial f(\x) = \nabla F(\x)^* \partial g(F(\x))$. This means for $\x\in V$, it holds that for any $\q \in \partial f(\x)$ there is some $u\in \partial g(F(\x))$ such that $\q = \nabla F(\x)^* u$. Fix $\bar \q \in \partial f(\bar \x)$. We want to show that there exist $\epsilon > 0$ and $\eta > 0$ such that for any $\x$ with $\|\x - \bar \x\| < \epsilon$ and $u\in \partial g(F(\x))$ with the property $\|\q - \bar\q\|=\|\nabla F(\x)^* u - \bar \q\| < \epsilon$ we have that $\|u\| < \eta$.
Since $g$ is convex it is locally Lipschitz on $\intr(\dom g)$. This means whenever $F(\bar \x) \in \intr (\dom g)$ there is $\epsilon >0$ sufficiently small such that due to continuity of $F$ we have $F(\x) \in \intr(\dom g)$ near $F(\bar \x)$ and there is some finite $\eta > 0$ such that $\|u\| < \eta$ for any $u \in \partial g(F(\x))$.
Now assume $F(\bar \x) \in \bdry(\dom g)$ and suppose that there exist sequences $\x^\nu \to \bar \x$ and \mynewline$u^\nu \in \partial g(F(\x^\nu))$ with $\nabla F(\x^\nu)^* u^\nu \to \bar \q$ and $\|u^\nu\| \to \infty$.
For a decomposition $u^\nu=u_0^\nu + u_i^\nu$ with $u_0^\nu\in\ker \nabla F(x^\nu)^*$ and $u_i^\nu\in\ran\nabla F(x^\nu)$, this yields $\|u_0^\nu\|\to\infty$. Through \cite[Proposition 8.12]{RoWe98} $u^\nu$ is in particular a regular subgradient of $g$ at $F(\x^\nu)$. Obviously, by possibly going to a subsequence, $u^\nu/\|u^\nu\|$ converges to a point on the unit circle, which, by definition, belongs to the horizon subgradient and, by \cite[Proposition 8.12]{RoWe98}, to $N_{\cl(\dom g)}(F(\bar x))$. Moreover, this point lies in $\ker\nabla F(\bar x)^*$, since $u_i^\nu/\|u^\nu\|\to 0$. This is a contradiction to the constraint qualification.
Let $\|\x-\bar \x\|<\epsilon$, $\|\x'-\bar \x\|<\epsilon$ and $\nabla F(\x)^* u = \q\in \partial f(\x)$ with $\|\q-\bar \q\|<\eps$ for some $u\in \partial g(F(\x))$. Due to the argument above we have $\|u\| \leq \eta$ and therefore also $\|u\|_1 \leq \gamma$ for some $\gamma>0$. Then, since $F$ is component-wise $L$-smad, thanks to Lemma~\ref{lem:full_descent}, we can make the following computation. We have for some $r\geq \gamma L$:
\begin{align*}
  f(\x') - f(\x) &= g(F(\x')) - g(F(\x)) \\
  &\geq \scal{u}{F(\x')-F(\x)}\\
  %&= \sum_{i=1}^n u_i\langle \nabla F_i(\x), \x' -\x \rangle \\
  %&\qquad+ \sum_{i=1}^n u_i(F_i(\x') - F_i(\x) - \langle \nabla F_i(\x), \x' -\x \rangle) \\
  &\geq \scal{u}{\nabla F(\x)(\x'-\x)} -\sum_{i=1}^n |u_i| L \D[\phi](\x',\x) \\
  &\geq \scal{u}{\nabla F(\x)(\x'-\x)} - \gamma L \D[\phi](\x',\x) \\
  &\geq \scal{\nabla F(\x)^*u}{\x'-\x} - r\D[\phi](\x',\x) \\
  &= \scal{\q}{\x'-\x} - r \D[\phi](\x',\x),
\end{align*}
which shows that $f$ is prox-regular at $\bar \x$ for $\bar \q$ relative to $\phi$.
\qed
\end{proof}

\begin{remark}
Note that the estimate in the proof also holds when each component function $F_i$ is $L$-smad relative to a potentially different $\phi_i$.
In addition, it should be noted that when $g$ is globally Lipschitz, which means in particular that $\dom g=\bR^n$, and $F$ is $L$-smad relative to $\phi$ on $V=\intr(\dom \phi)$, the composition $f=g \circ F$ is even relatively hypoconvex, i.e. $f+r\phi$ is convex on $\intr(\dom \phi)$ for some $r>0$ sufficiently large.
\end{remark}

Amenable functions whose representation $g\circ F$ involves a diffeomorphism $F$ have rich properties: As the following proposition shows, even for a prox-regular (outer) function $g$, the composition is also prox-regular.
\begin{proposition} \label{prop:amenable_nabla_phi}
Let $\map{f}{\bR^m}{\exR}$ be finite at $\bar\x$. Let $V \subset \bR^m$ be an open neighborhood of $\bar \x$ on which $f$ can be represented in the form $f=g\circ F$ for a $\mathcal{C}^1$ mapping $\map{F}{V}{\bR^m}$ and a function $\map{g}{\bR^m}{\exR}$. If $g$ is prox-regular at $F(\bar \x)$ for $\bar u \in \partial g(F(\bar \x))$, $\nabla F(\bar \x)$ is nonsingular and $\nabla F$ is Lipschitz on $V$, then $f$ is prox-regular at $\bar \x$ for $\bar \q = \nabla F(\bar \x)^{*} \bar u$.
\end{proposition}
\begin{proof}
Since $g$ is prox-regular at $F(\bar \x)$ for $\bar u \in \partial g(F(\bar \x))$ and due to the continuity of $F$, there exists a constant $r'> 0$ such that for any $\epsilon'> 0$ sufficiently small we have:
\begin{align*}
 g(F(\x'))\geq g(F(\x)) + \scal{ u}{F(\x')-F(\x)} - \frac {r'} 2\|F(\x') - F(\x)\|^2
\end{align*}
for $\|F(\bar \x) - F(\x')\| < \epsilon'$, $\|F(\bar \x) - F(\x)\| < \epsilon'$, $\|\bar u-u\| < \epsilon'$ with $u \in \partial g(F(\x))$ and $| g(F(\x))- g(F(\bar\x))| < \epsilon'$.
Since $F$ is locally Lipschitz, there exists $r''$ such that 
\begin{align*}
    \frac {r'} 2\|F(\x') - F(\x)\|^2 \leq \frac{r''} 2 \|\x' - \x\|^2.
\end{align*}

For the inner product, we use the fact that the component functions $F_i$ satisfy
\begin{align*}
      F_i(\x') - F_i(\x) = \int_{0}^{1}{ \langle \nabla F_i(\x+t(\x'-\x)), \x' - \x \rangle}{\,\mathrm{d}t} 
\end{align*}
and since $\nabla F$ is Lipschitz we have for some $L>0$
\begin{align*}
    \scal{u}{F(\x') - F(\x)} &=  \scal{u}{\nabla F(\x)(\x' - \x)} \\
    &\qquad+ \int_{0}^{1}{\big\langle u, \big(\nabla F(\x+t(\x'-\x)) - \nabla F(\x)\big) (\x' - \x)\big\rangle}{\,\mathrm{d}t} \\
    %&=  \scal{\nabla F(\x)^* u}{\x' - \x} + \int_{0}^{1}{\langle u, (\nabla F(\x+t(\x'-\x)) - \nabla F(\x)) (\x' - \x)\rangle}{\,\mathrm{d}t} \\
    &\leq \scal{\nabla F(\x)^* u}{\x' - \x} \\
    &\qquad+ \|u\| \int_{0}^{1}{\|\nabla F(\x+t(\x'-\x)) - \nabla F(\x)\|\cdot \|\x' - \x\|}{\,\mathrm{d}t} \\
    &\leq \scal{\nabla F(\x)^* u}{\x' - \x} + \|u\|\cdot\|\x'-\x\|^2 \frac{L}{2}. %\\\int_{0}^{1}{L t}{\,\mathrm{d}t} \\
    %&\leq \scal{\nabla F(\x)^* u}{\x' - \x} + \frac{r'''}{2} \|\x'-\x\|^2 
\end{align*}
  Combining the inequalities we obtain, since $u$ is bounded around $\bar u$, that for some $r>0$ we have
  \begin{align} \label{eq:ineq_pr}
 g(F(\x'))\geq g(F(\x)) + \scal{\nabla F(\x)^* u}{\x' - \x} - \frac {r} 2\|\x' - \x\|^2,
  \end{align}
 whenever $\|F(\bar \x) - F(\x')\| < \epsilon'$, $\|F(\bar \x) - F(\x)\| < \epsilon'$, $\|\bar u-u\| < \epsilon'$ with $u \in \partial g(F(\x))$ and $| g(F(\x))- g(F(\bar\x))| < \epsilon'$. 
 
 As $F$ is $\mathcal{C}^1$ with $\nabla F(\bar \x)$ nonsingular, in view of the inverse function theorem, we know that $F$ is invertible on a small neighborhood of $\bar \x$. % with inverse $F^{-1}$ being $\mathcal{C}^1$ and defined on some neighborhood U of $F(\bar \x)$ and $\nabla (F^{-1})(F(\x)) = (\nabla F(\x))^{-1}$.

  In view of \cite[Exercise 10.7]{RoWe98}, the chain rule holds on a neighborhood of $\bar \x$, i.e. we have $\nabla F(\x)^*\partial g(F(\x)) = \partial f(\x)$ when $\x$ near $\bar \x$. This means for such $\x$ and any $\q \in \partial f(\x)$ there exists $u \in \partial g(F(\x))$ such that $\q = \nabla F(\x)^*u$.

Let $\q \in \partial f(\x)$ near $\bar \q$ and $\x$ near $\bar \x$. Let $u \in \partial g(F(\x))$ such that \mynewline $\q = \nabla F(\x)^*u$. In view of the Lipschitz continuity of $\nabla F$, we can make the following computation:
\begin{align*}
\sigma_{\min}(\nabla F(\x))\|u - \bar u\| &\leq \|\nabla F(\x)^* u - \nabla F(\x)^*\bar u\| \\
%&=\|\nabla F^*(\x) u - \nabla F^*(\bar \x)\bar u + \nabla F^*(\bar \x)\bar u  - \nabla F^*(\x)\bar u\| \\
&\leq\|\nabla F(\x)^* u - \nabla F(\bar \x)^*\bar u \|+ \|\nabla F(\bar \x)^*\bar u  - \nabla F(\x)^*\bar u\| \\
%&\leq\|\q - \bar \q \|+\|\bar u\|\cdot  \|\nabla F(\bar \x)  - \nabla F(\x)\|  \\
&\leq\|\q - \bar \q \|+\|\bar u\| L \|\bar \x  - \x\|.
\end{align*}
Since $\nabla F(\x)$ is invertible we know that the smallest singular value $\sigma_{\min}(\nabla F(\x))$ of $\nabla F(\x)$ is positive. Then dividing the inequality by $\sigma_{\min}(\nabla F(\x))$ shows that for $\q \in \partial f(\x)$ near $\bar \q$ and $\x$ near $\bar \x$ we guarantee $u$ near $\bar u$.

Overall, this means we can find $\epsilon>0$ sufficiently small, such that whenever $\|\bar \x - \x\| < \epsilon$, $\|\bar \x - \x'\| < \epsilon$ and $\|\bar \q-\q\| < \epsilon$, $\q \in \partial f(\x)$ and $|f(\x) - f(\bar \x)| < \epsilon$, we guarantee via the continuity of $F$ that $\|F(\bar \x) - F(\x')\| < \epsilon'$, $\|F(\bar \x) - F(\x)\| < \epsilon'$, $| g(F(\x))- g(F(\bar\x))| < \epsilon'$ and $\|\bar u-u\| < \epsilon'$.
Then, in view of \eqref{eq:ineq_pr}, we have:
  $$
 f(\x')\geq f(\x) + \scal{\q}{\x' - \x} - \frac {r} 2\|\x' - \x\|^2.
  $$
  Since $g$ is in particular finite and locally lsc at $F(\bar \x)$, $f$ is finite and locally lsc at $\bar \x$.
We may conclude that $f$ is prox-regular at $\bar \x$ for $\bar \q$.
\qed
\end{proof}

A particularly interesting choice for $F$ in context of the right Bregman proximal mapping is $F=\nabla \phi^*$, for a Legendre function $\phi$:
\begin{corollary} \label{cor:prox_regularity_right}
Let $\phi \in \Gamma_0(\bR^m)$ be Legendre and $\map{f}{\bR^m}{\exR}$ be finite at $\bar \x$. Let $\phi \in \Gamma_0(\bR^m)$ be Legendre, very strictly convex and $\nabla^2 \phi$ locally Lipschitz at $\bar \x$. Then $f$ is prox-regular at $\bar \x$ for $\bar \q \in \partial f(\bar \x)$ if and only if $f \circ \nabla \phi^*$ is prox-regular at $\nabla \phi(\bar \x)$ for $\bar u = \nabla^2\phi^*(\nabla \phi(\bar \x))\bar \q \in \partial(f\circ \nabla \phi^*)(\nabla \phi(\bar \x))$.
 \end{corollary}
 \begin{proof}
 Since $\phi$ is very strictly convex, we know that $\nabla^2 \phi(\x)$ is positive definite for $\x \in \intr(\dom \phi)$ and therefore nonsingular. In view of Lemma~\ref{lem:inverse_hessian} we know that $\nabla^2 \phi^*(\x) = (\nabla^2 \phi(\nabla \phi^*(\x)))^{-1}$, which is locally Lipschitz as the composition of the inverse matrix map, $\nabla^2 \phi$ and $\nabla \phi^*$, all of which are locally Lipschitz, cf. \cite[Section 15; Exercise 22]{harville2001matrix}.
 The conclusion then follows from applying Proposition~\ref{prop:amenable_nabla_phi} to $f \circ \nabla \phi^*$ resp. $f=(f \circ \nabla \phi^*) \circ \nabla \phi$.
 \qed
 \end{proof}
  In particular, combining Lemma~\ref{lem:right_prox}, Theorem~\ref{thm:prox_regularity_equiv} and the Corollary~\ref{cor:prox_regularity_right} above we may guarantee the local single-valuedness of the right Bregman proximal mapping $\rprox{\phi}{f}{\lambda}$ of $f$ under prox-regularity of $f$ and very strict convexity of $\phi$.

The class of relatively amenable functions is a wide source of examples for relatively prox-regular functions:
\begin{example} \label{ex:amenable}
Choose $\map{f}{\bR^2}{\exR}$ with $f(\x_1, \x_2) = g(F(\x_1, \x_2))$ for $\map{g}{\bR}{\exR}$ with $g := \delta_{\bR_{\leq0}}$ and $\map{F}{\bR^2}{\bR}$ with $F(\x_1, \x_2) = 2\x_1^2 - 3|\x_1|^{1.1} - \x_2$. Choose $\map{\phi}{\bR^2}{\bR}$ with $\phi(\x_1, \x_2)= \x_1^2 + |\x_1|^{1.1} + \x_2^2$. Then, clearly, $F$ is $L$-smad relative to $\phi$ for $L=3$. Since $\nabla F(0)= (0, -1)$ is full rank, $f$ is relatively amenable at $0$ and in view of Proposition~\ref{prop:amenable_implies_pr}, relatively prox-regular at $0$. Note that $f$ is the indicator function of the epigraph of the nonconvex function \mynewline $h(\x) = 2\x^2 - 3|\x|^{1.1}$ and therefore neither hypoconvex relative to $\phi$ nor classically prox-regular at $0$.
\end{example}
The above example is illustrated in Figure~\ref{fig:bergman_proj}.

\subsection{Gradient Formulas for Bregman--Moreau Envelopes} \label{sec:gradient_formulas}
So far we know that relative prox-regularity provides us with a sufficient condition for the local single-valuedness of the left and right Bregman proximal mapping. This in turn allows us to guarantee that the Bregman envelope functions are locally $\mathcal{C}^1$ providing an explicit formula for their gradients, which involves the corresponding Bregman proximal mappings. The formulas for both the left and the right envelope have been proven previously in the convex setting \cite[Proposition 3.12]{bauschke2006joint} and for the left envelope in a more general relatively hypoconvex setting \cite[Corollary 3.1]{kan2012moreau}.

The following proposition provides us with an explicit formula for the gradient of the composition $\lenv{\phi}{f}{\lambda} \circ \nabla \phi^*$. The gradient formulas of both the left and right envelope are direct consequences of this underlying formula.
\begin{proposition} \label{prop:gradient_left_phi}
Let $\phi \in \Gamma_0(\bR^m)$ be Legendre and super-coercive and the function $\map{f}{\bR^m}{\exR}$ be proper lsc and prox-bounded with threshold $\lambda_f$. 
Let $f$ be relatively prox-regular at $\bar{\x} \in \intr (\dom \phi) \cap \dom f$ for $\bar{\q}\in\partial f(\bar{\x})$. 

If $\lambda \in \left]0,\lambda_f\right[$ is sufficiently small, we have that $\lenv{\phi}{f}{\lambda} \circ \nabla \phi^*$ is $\mathcal{C}^1$ around 
\begin{align*}
\bar{\y} := \nabla\phi(\bar{\x}) + \lambda \bar{\q},
\end{align*}
with 
\begin{equation} \label{eq:gradient_formula_left_phi}
\nabla \big(\lenv{\phi}{f}{\lambda} \circ \nabla \phi^* \big)(\y)=\frac{1}{\lambda}\big(\nabla \phi^*(\y)-\lprox{\phi}{f}{\lambda}(\nabla\phi^*(\y)) \big),
\end{equation}
and $\y$ sufficiently close to $\bar{\y}$.
If, furthermore, $\phi$ is very strictly convex, then $\nabla (\lenv{\phi}{f}{\lambda} \circ \nabla \phi^*)$ is Lipschitz continuous on a neighborhood of $\bar{\y}$.
\end{proposition}

\begin{proof}
Let $\lambda \in \left]0,\lambda_f\right[$. In view of the relative prox-boundedness of $f$, we know due to the continuity properties summarized in Lemma~\ref{lem:continuity_left}, the super-coercivity of $\phi$ and the continuity of $\nabla \phi^*$ and the fact that $\dom \phi^* = \bR^m$, cf. Lemma~\ref{lem:legendre_props}, that for any $\y \in \bR^m$ there is a neighborhood $V \subset \bR^m$ of $\y$, sufficiently small along with a compact set $Z \subset \dom \phi$ such that for any $\y' \in V$ we can write $-(\lenv{\phi}{f}{\lambda} \circ \nabla \phi^*)(\y')=\max_{\x \in Z} h(\x,\y')$, for
$$
h(\x,\y'):=-f(\x) -\frac{1}{\lambda}\D[\phi](\x,\nabla \phi^*(\y')) = -f(\x) -\frac{1}{\lambda}\big(\phi(\x) + \phi^*(\y') - \langle \y', \x \rangle\big) 
$$
and $\lprox{\phi}{f}{\lambda}(\nabla\phi^*(\y')) \subset Z$. Clearly, $h(\x,\cdot)$ is $\mathcal{C}^1$ as $\phi^*$ is $\mathcal{C}^1$ on $\dom \phi^* = \bR^m$ with $h(\x, \y)$ and $\nabla_\y h(\x,\y)=-(1/\lambda)(\nabla \phi^*(\y)-\x)$ both depending continuously on $(\x,\y) \in Z \times V$. Hence $h$ is lower-$\mathcal{C}^1$, cf. \cite[Definition 10.29]{RoWe98}, and therefore we can invoke \cite[Theorem 10.31]{RoWe98}, to obtain that
\begin{align} \label{eq:neg_subdiff_left_phi}
\partial\big(-\lenv{\phi}{f}{\lambda} \circ \nabla \phi^* \big)(\y)  &= \con \big\{ \nabla_\y h(\x, \y) : \x \in \lprox{\phi}{f}{\lambda}(\nabla\phi^*(\y)) \big\} \notag\\
&= -\frac1\lambda\Big( \nabla \phi^*(\y)-{\con} \big(\lprox{\phi}{f}{\lambda}(\nabla \phi^*(\y))\big) \Big).
\end{align}
Due to the assumptions we can invoke Theorem~\ref{thm:prox_regularity_equiv}(iii) and assert that for $\lambda \in \left]0,\lambda_f\right[$ sufficiently small $\lprox{\phi}{f}{\lambda}\circ \nabla \phi^*$ is singled-valued at $\y$ near \mynewline$\bar\y = \nabla\phi(\bar{\x}) + \lambda \bar{\q}$.
In view of Equation~\eqref{eq:neg_subdiff_left_phi}, this means that $\partial(-\lenv{\phi}{f}{\lambda} \circ \nabla \phi^*)$ is single-valued around $\bar{\y}$. Through \cite[Corollary 9.19]{RoWe98} we obtain that \mynewline $-\lenv{\phi}{f}{\lambda} \circ \nabla \phi^*$ is $\mathcal{C}^1$ around $\bar{\y}$ with 
$$
\frac1\lambda\big(\nabla \phi^*(\y)-\lprox{\phi}{f}{\lambda}(\nabla\phi^*(\y))\big) = \nabla \big(\lenv{\phi}{f}{\lambda} \circ\nabla \phi^*\big)(\y).
$$
If, furthermore, $\phi$ is very strictly convex, we know due to Corollary~\ref{cor:prox_lipschitz} that $\lprox{\phi}{f}{\lambda}$ is locally Lipschitz at $\nabla\phi^*(\bar{\y})$. Then $\nabla (\lenv{\phi}{f}{\lambda} \circ \nabla \phi^*)$ is locally Lipschitz at $\bar{\y}$ as a composition resp. sum of locally Lipschitz maps.
\qed
\end{proof}

\begin{corollary} \label{cor:gradient_env_left}
Let $\phi \in \Gamma_0(\bR^m)$ be Legendre, super-coercive and $\mathcal{C}^2$ on $\intr(\dom \phi)$ and $\map{f}{\bR^m}{\exR}$ be proper lsc and prox-bounded with threshold $\lambda_f$. 
Let $f$ be relatively prox-regular at $\bar{\x} \in \intr (\dom \phi) \cap \dom f$ for $\bar{\q}\in\partial f(\bar{\x})$. If $\lambda \in \left]0,\lambda_f\right[$ is sufficiently small we have that $\lenv{\phi}{f}{\lambda}$ is $\mathcal{C}^1$ around
\begin{align*}
\bar{\y} := \nabla\phi^*(\nabla\phi(\bar{\x}) + \lambda \bar{\q}),
\end{align*}
with 
\begin{equation} \label{eq:gradient_formula}
\nabla \lenv{\phi}{f}{\lambda}(\y)=\frac1\lambda \nabla^2\phi(\y)\big(\y-\lprox{\phi}{f}{\lambda}(\y)\big),
\end{equation}
and $\y$ sufficiently close to $\bar{\y}$.
If, furthermore, $\phi$ is very strictly convex and $\nabla^2 \phi$ is locally Lipschitz on $\intr(\dom \phi)$, then $\nabla \lenv{\phi}{f}{\lambda}$ is Lipschitz continuous on a neighborhood of $\bar{\y}$.
\end{corollary}
\begin{proof}
The result follows from the identity $\lenv{\phi}{f}{\lambda}= (\lenv{\phi}{f}{\lambda} \circ \nabla \phi^*) \circ \nabla \phi$ via the chain rule and Proposition~\ref{prop:gradient_left_phi}: Then we have for $\y$ near $\bar \y$ that 
\begin{align*}
\nabla \lenv{\phi}{f}{\lambda}(\y) &= \nabla \big((\lenv{\phi}{f}{\lambda} \circ \nabla \phi^*) \circ \nabla \phi \big)(\y) \\
&= \nabla^2 \phi(\y) \cdot \nabla \big(\lenv{\phi}{f}{\lambda} \circ \nabla \phi^*\big)(\nabla \phi(\y)) \\
&= \frac{1}{\lambda} \nabla^2 \phi(\y) \big(\nabla \phi^*(\nabla \phi(\y))-\lprox{\phi}{f}{\lambda}(\nabla\phi^*(\nabla \phi(\y)))\big)\\
&= \frac{1}{\lambda} \nabla^2 \phi(\y) \big(\y-\lprox{\phi}{f}{\lambda}(\y)\big).
\end{align*}
If, furthermore, $\phi$ is very strictly convex and $\nabla^2 \phi$ is locally Lipschitz, clearly, $\nabla \lenv{\phi}{f}{\lambda}$ is locally Lipschitz at $\bar \y$ as it is given as the product of two locally Lipschitz maps.
\qed
\end{proof}

In view of Lemma~\ref{lem:right_prox}, the right Bregman envelope involves the expression $\lenv{\phi^*}{(f \circ \nabla \phi^*)}{\lambda} \circ \nabla \phi$. This allows us to invoke the proposition above to derive a gradient formula for the right envelope.
\begin{corollary} \label{cor:gradient_env_right}
Let $\phi \in \Gamma_0(\bR^m)$ be Legendre with $\dom \phi = \bR^m$ and the function $\map{f}{\bR^m}{\exR}$ be proper lsc and right prox-bounded with threshold $\lambda_{f}$. For $\bar \x \in \dom f$ let $f \circ \nabla \phi^*$ be prox-regular relative to $\phi^*$ at $\nabla \phi(\bar{\x})$ for \mynewline$\bar{\q}\in\partial (f \circ \nabla \phi^*)(\nabla\phi(\bar{\x}))$. If $\lambda \in \left]0,\lambda_f\right[$ is sufficiently small we have that $\renv{\phi}{f}{\lambda}$ is $\mathcal{C}^1$ around
\begin{align*}
\bar\y=\bar{\x} + \lambda \bar{\q}
\end{align*}
with 
\begin{align} \label{eq:gradient_formula_right}
\nabla \renv{\phi}{f}{\lambda}(\y)&=\frac1\lambda \big(\nabla \phi(\y)-\nabla \phi(\rprox{\phi}{f}{\lambda}(\y))\big)\notag \\
&= \frac1\lambda \big(\nabla \phi(\y)-\lprox{\phi^*}{(f\circ\nabla \phi^*)}{\lambda}(\nabla \phi(\y))\big),
\end{align}
and $\y$ sufficiently close to $\bar{\y}$.
If, furthermore, $\phi$ is very strictly convex, $\nabla \renv{\phi}{f}{\lambda}$ is Lipschitz continuous on a neighborhood of $\bar{\y}$.
\end{corollary}
\begin{proof}
In view of Lemma~\ref{lem:legendre_props}, $\phi^*$ is super-coercive. Then the result follows from the identities 
$$
\renv{\phi}{f}{\lambda}= \lenv{\phi^*}{(f \circ \nabla \phi^*)}{\lambda} \circ \nabla \phi
$$
and 
$$
\rprox{\phi}{f}{\lambda}(\y) = \nabla \phi^*\big(\lprox{\phi^*}{(f \circ \nabla \phi^*)}{\lambda} (\nabla \phi(\y))\big),
$$
cf. Lemma~\ref{lem:right_prox} as well as Lemma~\ref{lem:continuity_right} and Proposition~\ref{prop:gradient_left_phi} applied to \mynewline$\lenv{\phi^*}{(f \circ \nabla \phi^*)}{\lambda} \circ \nabla \phi$.
\qed
\end{proof}

Note that when $\phi$ is very strictly convex and in addition $\nabla^2 \phi$ is Lipschitz at $\bar \x$, in view of Corollary~\ref{cor:prox_regularity_right} and Proposition~\ref{prop:prox_regular_very_strictly_convex}, the relative prox-regularity assumption on $f \circ \nabla \phi^*$ is equivalent to classical prox-regularity of $f$ at $\bar \x$ for $\nabla^2 \phi(\bar \x)\bar \q \in \partial f(\bar\x)$.

The gradient formula of the right Bregman envelope provides us with an explicit sufficient condition for the local $\mathcal{C}^1$ property of the right Bregman distance function of a convex set, which, in view of \cite[Theorem 6.6]{bauschke2009bregman} and \cite[Example 7.5]{bauschke2009bregman}, does not hold globally in general. To illustrate this we revisit and extend \cite[Example 7.5]{bauschke2009bregman}.
\begin{example}
Define the Legendre function $\map{\phi}{\bR^2}{\bR}$ as 
$$
\phi(\x_1, \x_2) := \exp(\x_1) + \exp(\x_2).
$$ 
Then the convex conjugate is given as
$$
\phi^*(\y_1, \y_2) = \begin{cases} \y_1 \log(\y_1) - \y_1 + \y_2 \log(\y_2) - \y_2 & \text{if $\y_1 \geq 0, \y_2 \geq 0$,} \\
+\infty, & \text{otherwise.}
\end{cases}
$$
Define $f:=\delta_C$ for
$$
C:=\{(\x,2\x) : \x \in [0,1]\}.
$$
Then $f \circ \nabla \phi^* = \delta_{\nabla \phi(C)}$ for $\nabla \phi(C) = \{(\exp(\x),\exp(2\x)) : \x \in [0,1]\}$ which is obviously a nonconvex set. Let $\bar{\x} \in C$. Since $\phi^*$ is $\mathcal{C}^3$ on $\intr(\dom \phi^*)$ and very strictly convex and therefore $\nabla^2 \phi^*$ is full rank at $\nabla\phi(\bar \x)$ and \mynewline$(f \circ \nabla \phi^*)(\nabla \phi(\bar \x))=f(\bar \x)$ is finite we have that $f \circ \nabla \phi^*$ is strongly amenable at $\nabla\phi(\bar \x)$. In view of \cite[Proposition 13.32]{RoWe98} $f \circ \nabla \phi^*$ is also prox-regular at $\nabla\phi(\bar \x)$. Let $\bar{\q} \in \partial (f \circ \nabla \phi^*)(\nabla \phi(\bar \x))$, i.e. $\bar \q$ is a limiting normal of $\nabla \phi(C)$ at $\nabla\phi(\bar \x)$ and, in view of the prox-regularity of $f \circ \nabla \phi^*$ at $\nabla\phi(\bar \x)$, even a proximal normal of $\nabla \phi(C)$ at $\nabla\phi(\bar \x)$. In view of Proposition~\ref{prop:prox_regular_very_strictly_convex}, $f \circ \nabla \phi^*$ is also prox-regular relative to $\phi^*$ at $\nabla \phi(\bar \x)$. Then we can invoke Corollary~\ref{cor:gradient_env_right} to assert, that for $\lambda > 0$ being sufficiently small, the right Bregman distance function $\renv{\phi}{f}{\lambda} = \renv{\phi}{\delta_C}{1}$ is $\mathcal{C}^1$ around the point $\bar{\y}=\bar \x + \lambda \bar \q$.
\end{example}
In view of Corollary~\ref{cor:prox_regularity_right}, the local $\mathcal{C}^1$ property of the right Bregman distance function $\renv{\phi}{\delta_C}{1}$ even holds for nonconvex $C$ with $\delta_C$ prox-regular.

We conclude this section by providing an interesting additional (global) regularity property of the Bregman envelope function: Both the left envelope $\lenv{\phi}{f}{\lambda} \circ \nabla \phi^*$ and the right envelope $\renv{\phi}{f}{\lambda}$ have the one-sided $L$-smad property relative to $\phi^*$ resp. $\phi$, and therefore yield promising candidates for optimization with Bregman proximal gradient methods \cite{BBT2016,BSTV2018}, cf. Section~\ref{sec:algorithm}.

\begin{proposition} \label{prop:lsmad_env}
Let $\phi \in \Gamma_0(\bR^m)$ be Legendre and $f:\bR^m\to\exR$ be proper lsc. Then it holds that:
\begin{enumerate}
\item[(i)] If $\dom f \cap \dom \phi$ is non-empty, $f$ is prox-bounded relative to $\phi$ with threshold $\lambda_f$ and $\phi$ is super-coercive, then for any $\lambda \in \left]0,\lambda_f\right[$,
$$
\frac1\lambda \phi^* - \lenv{\phi}{f}{\lambda} \circ \nabla \phi^* = {\left( f + \frac1\lambda \phi \right)}^*{\left( \frac\cdot\lambda \right)}
$$
is proper, lsc and convex.
\item[(ii)] If $f$ is right prox-bounded relative to $\phi$ with threshold $\lambda_f$ and $\dom \phi= \bR^m$, then for any $\lambda \in \left]0,\lambda_f\right[$,
$$
 \frac1\lambda \phi - \renv{\phi}{f}{\lambda} = {\left( f \circ \nabla \phi^* + \frac1\lambda \phi^* \right)}^*{\left( \frac\cdot\lambda \right)}
$$
is proper, lsc and convex.
\end{enumerate}
\end{proposition}
\begin{proof}
Let $f$ be prox-bounded relative to $\phi$ super-coercive with threshold \mynewline $\lambda_f>0$ and let $\lambda \in \left]0,\lambda_f\right[$. 
From \cite[Theorem 2.4]{kan2012moreau} we obtain that we have for all $\y \in \bR^m$:
$$
\frac1\lambda \phi^*(\y) - {\left( f + \frac1\lambda \phi \right)}^*{\left( \frac\y\lambda \right)} = \big(\lenv{\phi}{f}{\lambda} \circ \nabla \phi^*\big)(\y).
$$
Then, in view of Lemma~\ref{lem:continuity_left}, we know that $\lenv{\phi}{f}{\lambda}$ is proper and continuous on $\intr(\dom \phi) = \ran \nabla \phi^*$, which means in particular, thanks to Lemma~\ref{lem:legendre_props}(iv), $\dom(\lenv{\phi}{f}{\lambda} \circ \nabla \phi^*) = \bR^m$. Furthermore, in view of Proposition~\ref{prop:relative_prox_bounded_equivalence}(ii), \mynewline $f + (1/\lambda) \phi$ is bounded below and proper. Then, clearly, $\con (f + (1/\lambda) \phi)$ is proper, and in view of \cite[Theorem 11.1]{RoWe98}, $\left( f + (1/\lambda) \phi \right)^*\left( \cdot/\lambda \right)$ is proper, lsc and convex. Since in view of Lemma~\ref{lem:legendre_props}(iv) also $\dom \phi^* =\bR^m$ we can reorder the terms and obtain that
$$
\frac1\lambda \phi^*(\y) - \big(\lenv{\phi}{f}{\lambda} \circ \nabla \phi^*\big)(\y) = {\left( f + \frac1\lambda \phi \right)}^*{\left( \frac\y\lambda \right)},
$$
and the assertion follows.

Part (ii) follows from a similar argument invoking \cite[Proposition 2.4(ii)]{bauschke2017regularizing} and the observation that right prox-boundedness of $f$ relative to $\phi$ implies prox-boundedness of $f \circ \nabla \phi^*$ relative to $\phi^*$.
\qed
\end{proof}

\section{Algorithmic Implications of Relative Prox-regularity} \label{sec:algorithm}
\subsection{Example of a Simple Bregman Proximal Mapping}

We present an analytically solvable Bregman proximal mapping for the relatively prox-regular function $(1/p) \abs{x}^p$ for $p\in \left]0,1\right[$. While the function is also prox-regular, the classical proximal mapping cannot be solved analytically, except for $p=1/2$. For each $p\in \left]0,1\right[$, we define a Legendre function $\phi$ relative to which $(1/p) \abs{x}^p$ is prox-regular and the left Bregman proximal mapping can be solved easily. This example is potentially interesting for applications that involve optimization with sparsity regularization, as for example in compressed sensing.
\begin{example}
Let $f:\bR \to \bR$ with $f(x) = (1/p) \abs{\x}^p$ with $p\in \left]0,1\right[$ and choose $\phi(\x) = (1/q) \abs{\x}^q$ with $q>1$. For some $\y \in \bR^m$ we seek a closed form solution of the left Bregman proximal mapping
$$
\lprox{\phi}{f}{\lambda}(\y) = \argmin_{x \in \bR} \frac 1p \abs{x}^p + \frac1\lambda \D[\phi](x,\y) = \argmin_{\x \in \bR} \frac 1p \abs{x}^p + \frac 1{q\lambda} \abs{x}^q - c x,
$$
for $c:=(1/\lambda) \sign(\y)\abs{\y}^{q-1}$. Let $\bar \x \in \lprox{\phi}{f}{\lambda}(\y) $. Note that
$$
  \partial f(\x) = \begin{cases} 
                  \sign(\x) \abs{\x}^{p-1} & \text{if $x\neq 0$}, \\
                  \bR, & \text{otherwise}.
                  \end{cases}
$$
and 
$
  \phi'(\x) = \sign(\x)\abs{\x}^{q-1}.$
Then, the first order necessary optimality condition is given as follows:
$$
0 \in \partial f(\bar \x) + \frac{1}{\lambda} \phi'(\bar \x) - c = \begin{cases}
\bar \x^{p-1}  + \frac{1}{\lambda}\bar \x^{q-1} - c & \text{if $ \bar \x>0$,} \\
\bR & \text{if $\bar \x=0$,} \\
-\bar \x^{p-1}  - \frac{1}{\lambda} \bar \x^{q-1} - c, & \text{otherwise.}
\end{cases}
$$
This shows that the left Bregman proximal mapping can be evaluated by checking the three conditions individually and combining the minimum objective solutions. Note that the first and the last condition are exclusive while the first two and the last two conditions can potentially be satisfied simultaneously. Indeed, the Bregman proximal mapping of the given $f$ can be multivalued. Assume that $\bar \x > 0$ as the other case follows analogously. I.e. we seek a point $\bar \x > 0$ that satisfies
$$
  \bar \x^{p-1}\left( 1 + \frac{1}{\lambda} \bar \x^{q-p} - c \bar \x^{1-p} \right) = 0.
$$
Let $\alpha\in\set{2,3,4,\ldots}$ and choose $q$ according to the following condition: 
$$
\frac{q-p}{1-p} = \alpha,
$$
which is equivalent to $q = \alpha + (1-\alpha) p$. Now, the substitution 
$$
\bar \x^{1-p} = u \iff \bar \x = u^{1/(1-p)}
$$
leads to the following root-finding problem
$$
  u^{-1}\left(1 + \frac1\lambda u^{\frac{q-p}{1-p}} - c u \right) = 0
  \quad \iff \quad %\Leftrightarrow \quad 
  1 + \frac1\lambda u^\alpha - c u = 0,
$$
which can be solved analytically (at least) for $\alpha\in\set{2,3,4}$. Verification that $f$ is relatively prox-regular is yet to be performed. Let $\bar \x>0$. We can choose $\eps>0$ such that the $\eps$-ball around $\bar \x$ lies in $\bR_{>0}$. Then we find $r>0$ sufficiently large such that for all $\x\in \bR$ with $|x-\bar\x | < \epsilon$ the second order derivative of $f + r \phi$ at $\x$, given as
$(f + r\phi)''(\x) =(1/(p-1)) \x^{p-2} + r (1/(q-1)) \x^{q-2} \geq 0$,
is nonnegative, which is asserted for 
$$
r\geq \frac{q-1}{1-p} \cdot\inf \left\{ \x^{(p-2)/(q-2)} : |\x - \bar x| < \epsilon \right\}> 0,
$$
since $(q-1)/(1-p) > 0$. This implies that $f + r\phi$ is convex on the open $\eps$-ball around $\bar \x$ and therefore $f$ is relatively prox-regular at $\bar \x$.
The case $\bar \x<0$ follows by symmetry. Now, we choose $\bar \x = 0$ and fix $\bar \q \in \partial f(0) = \bR$. Since $\lim_{x \to 0,\,\x\neq0} |f'(x)| \to \infty$ we can find $\epsilon$ sufficiently small such that the graph of the $\epsilon$-localization $T$ of $\partial f$ around $(\bar\x,\bar\q)$ degenerates to 
$$
\gph T = \left\{ (\bar \x,\q) : |\q - \bar \q| < \epsilon \right\}.
$$
Relative prox-regularity of $f$ at $\bar \x = 0$ for $\bar \q$ is then asserted by verifying the subgradient inequality \eqref{eq:prox_regularity} for all $(\x, \q) \in \gph T$.
Indeed, we can find $\epsilon > 0$, such that for all such $|\q - \bar \q| < \epsilon$ we have $f(\x') \geq \q \x'$ for all $ |\x' - \bar x| < \epsilon$, which shows that $f$ is relatively prox-regular also at $0$.

\end{example}

\subsection{Optimization Algorithms}
Let $\phi \in \Gamma_0(\bR^m)$ be a Legendre function. We denote by $C:= \intr(\dom \phi)$. We are interested in the following optimization problem,
\begin{align} \label{eq:model_left}
\text{minimize} \left\{F_\lambda(\w, \x) \equiv f(\x)  + \frac{1}{\lambda} \D[\phi](\x,\nabla\phi^*(A(\w))) + g(\w) : (\w, \x)\in \bR^n \times C \right\},
\end{align}
where $f:\bR^m \to \exR$ and $g:\bR^n \to \exR$ are proper lsc functions and $A: \bR^n \to \bR^m$ is an optional linear map.
Via inf-projection with respect to $\x$, the model is equivalent to the left Bregman relaxation:
\begin{align} \label{eq:model_left_envelope}
\text{minimize} \left\{ \big(\lenv{\phi}{f}{\lambda} \circ \nabla\phi^* \circ A\big)(\w) + g(\w) : \w\in \bR^n \right\}.
\end{align}
We are particularly interested in finding stationary points of the lower problem by a proximally regularized alternating minimization strategy applied to the upper problem.

\subsubsection{Convergence with Bregman Proximal Regularization}
The first algorithm we consider is a variant of alternating minimization of model~\eqref{eq:model_left}, see Algorithm~\ref{alg:alternating_minimization_left_proximal}, which involves a proximal regularization of both variables, similar to proximal alternating minimization \cite{attouch2010proximal}.

\begin{figure}[h]
\centering
\fbox{
\begin{minipage}{0.83\linewidth}
\begin{algorithm}[Bregman Proximal Alternating Minimization] \label{alg:alternating_minimization_left_proximal}
{\ }
Choose appropriate Legendre functions $\sigma \in \Gamma_0(\bR^m)$ and $\omega \in  \Gamma_0(\bR^n)$ with $\dom\omega =\bR^n$ and $\dom \sigma \supseteq \dom \phi$ and initialize $\x^0 \in \intr(\dom \phi)$ and $\w^0 \in \bR^n$.
For $t=1, 2, \dots$ do
\begin{align}
\x^{t+1} &:= \argmin_{\x \in \bR^m} ~ F_\lambda(\w^t, \x) + \D[\sigma](\x, \x^t), \\
\w^{t+1} &:= \argmin_{\w \in \bR^n} ~ F_\lambda(\w, \x^{t+1}) + \D[\omega](\w, \w^t) \label{eq:w_update}.
\end{align}
\end{algorithm}
\end{minipage}
}
\end{figure}
Note that the $\w$-update is in general a difficult problem. We may therefore replace the coupling function $\D[\phi](\x,\nabla\phi^*(A(\w)))$ with a proximal linearization as in \emph{proximal alternating linearized minimization} (PALM) \cite{bolte2014proximal}, which is captured in the Bregman proximal term $\D[\omega](\w,\w^t)$ in our formulation. 
E.g. let $\phi^*$ be classically $L$-smooth. Since, in addition, $A$ is linear, 
$$
\D[\phi](\x^{t+1},\nabla\phi^*(A(\w))) = \phi(\x^{t+1}) +  \phi^*(A(\w)) - \langle  \w, A^*\x^{t+1} \rangle
$$
is guaranteed to be $L\|A\|^2$-smooth in $\w$ and we may choose 
$$
\omega(\w):= \frac{M}{2\lambda}\|\w\|^2 - \frac{1}{\lambda}\phi^*(A(\w)),
$$
for $M > L\|A\|^2$. Then the $\w$-update \eqref{eq:w_update} becomes a classical proximal gradient step on $F_\lambda$ as in PALM:
$$
\w^{t+1} = \argmin_{\w \in \bR^n} ~ g(\w) + \frac{1}{\lambda}\langle \w, A^*(\nabla \phi^*(A\w^t) - \x^{t+1}) \rangle + \frac{M}{2\lambda}\|\w -\w^t\|^2.
$$

Remarkably, prox-regularity as a stability condition allows us to interpret the limit point $\w^*$ as a stationary point of the regularized problem \eqref{eq:model_left_envelope}, even though the algorithm performs proximally regularized $\x$-updates and the problem is nonconvex. A similar ``translation of stationarity'' has been oberserved previously in \cite{LWC18,laude-wu-cremers-aistats-19} for the classical Moreau envelope and an anisotropic generalization of the former.
\begin{theorem} \label{thm:convergence}
Let $\phi, \sigma \in \Gamma_0(\bR^m)$, $\omega \in \Gamma_0(\bR^n)$ be Legendre with $\dom\omega =\bR^n$ and $\dom \sigma \supseteq \dom \phi$ and $\phi$ be super-coercive. Let $f:\bR^m \to \exR$ and $g:\bR^n \to \exR$ be proper lsc, $\dom f \cap \dom \phi$ be non-empty and let the qualification condition \eqref{eq:qualification_interior} hold. Let $\lambda >0$ and $F_\lambda:\bR^n \times \bR^m \to \exR$ be coercive. Then any limit point $(\w^*, \x^*)$ of the sequence of iterates $\{\w^t, \x^t\}_{t \in \bN}$ produced by Algorithm~\ref{alg:alternating_minimization_left_proximal} is a stationary point of $F_\lambda$, i.e.
$$
0 \in \partial F_\lambda(\w^*, \x^*),
$$
and in particular $\x^* \in \intr(\dom \phi)$ and
\begin{align}
0 &\in \partial f(\x^*) + \frac{1}{\lambda} (\nabla \phi(\x^*) - A(\w^*)), \label{eq:optimality_1}\\
0 &\in \partial g(\w^*) + \frac{1}{\lambda} A^* (\nabla \phi^*(A(\w^*)) - \x^*) \label{eq:optimality_2}.
\end{align} 
If, furthermore, $f$ is prox-regular relative to $\phi$ at $\x^*$ for $\q^* \in \partial f(\x^*)$, $\lambda >0$ is chosen to be sufficiently small and it holds that $A(\w^*)=\nabla \phi(\x^*) + \lambda \q^*$, then $\w^*$ is also a stationary point of the left Bregman relaxation \eqref{eq:model_left_envelope}, i.e. in particular we have:
$$
0 \in \partial \big(\lenv{\phi}{f}{\lambda} \circ \nabla\phi^* \circ A + g\big)(\w^*).
$$
\end{theorem}
\begin{proof}
The first part of the proof is standard. For the sake of self-containedness we provide a proof in the Appendix. %~\ref{sec:convergence_proof}.

For the second part, it should first be noted that since $F_\lambda$ is coercive and $g,f$ are proper lsc, it holds that
$$
\inf_{\x \in \bR^m} f(\x) + g(\w^*) +\frac1\lambda\D[\phi](\x, A(\w^*)) > -\infty
$$
and therefore $f$ is prox-bounded relative to $\phi$ with some threshold $\lambda_f \geq \lambda$. In view of \eqref{eq:optimality_1}, the identity $A(\w^*)=\nabla \phi(\x^*) + \lambda \q^*$ yields 
$$
\q^*=\frac1\lambda (A(\w^*) - \nabla \phi(\x^*)) \in \partial f(\x^*).
$$
Since by assumption $f$ is prox-regular at $\x^*$ for $\q^*\in \partial f(\x^*)$ we can invoke Theorem~\ref{thm:prox_regularity_equiv} and Proposition~\ref{prop:gradient_left_phi} and prove that if the chosen $\lambda\in\left]0,\lambda_f\right[$ is sufficiently small we have that $\x^* = \lprox{\phi}{f}{\lambda}(A(\w^*))$ and $\lenv{\phi}{f}{\lambda}$ is $\mathcal{C}^1$ around $A(\w^*)=\nabla \phi(\x^*) + \lambda \q^*$ with
$$
 \frac{1}{\lambda} A^* (\nabla \phi^*(A(\w^*)) - \x^*) =  A^* \nabla \big(\lenv{\phi}{f}{\lambda} \circ \nabla \phi^*\big)(A(\w^*)).
$$
Combining this with \eqref{eq:optimality_2} yields
$$
0 \in \partial g(\w^*) +  A^* \nabla \big(\lenv{\phi}{f}{\lambda} \circ \nabla \phi^*\big)(A(\w^*)).
$$
In view of \cite[Exercise 8.8(c)]{RoWe98}, we get the conclusion.
\qed
\end{proof}

We conclude this section with the remark that one can derive analogous results for the right Bregman envelope starting from the problem
\begin{align} \label{eq:model_right}
\text{minimize} \left\{H_\lambda(\y, \x) \equiv  f(\x) + \frac{1}{\lambda} \D[\phi](\y, \x) + g(\y): (\y, \x)\in \bR^m \times \bR^m \right\},
\end{align}
with $\dom \phi =\bR^m$. In this case, we aim to find stationary points of
\begin{align} \label{eq:model_right_envelope}
\text{minimize} \left\{ \renv{\phi}{f}{\lambda}(\y) + g(\y) : \y\in \bR^m \right\},
\end{align}
via alternating minimization of the upper problem.
%with $\dom \phi=\bR^m$. Then via a similar argument a right Bregman relaxation:

\subsubsection{Local Convergence with Partial Bregman Proximal Regularization}
In this section we consider a variant of Algorithm~\ref{alg:alternating_minimization_left_proximal}, where we leave out proximal regularization of the $\x$-update, i.e. $\sigma \equiv 0$.
As a short computation reveals, when the gradient formula for the envelope $\lenv{\phi}{f}{\lambda} \circ \nabla \phi^*$ holds (which happens to be true locally whenever $f$ is relatively prox-regular around the limit point and relatively prox-bounded and $\lambda>0$ is sufficiently small) and $A=I$ is chosen to be the identity, we can rewrite the algorithm as the following Bregman proximal gradient update:
$$
\w^{t+1} = \argmin_{\w \in \bR^m} ~g(\w)+ \big\langle \nabla \big(\lenv{\phi}{f}{\lambda} \circ \nabla \phi^* \big)(\w^t), \w - \w^t \big\rangle + \D[\frac{1}{\lambda}\phi^* + \omega](\w, \w^t).
$$
Analogously, alternating minimization of \eqref{eq:model_right} with a proximal regularization of the $y$-update yields the following Bregman proximal gradient update involving the right Bregman envelope (assuming the gradient formula for the right envelope holds):
$$
\y^{t+1} = \argmin_{\y \in \bR^m} ~g(\y)+ \big\langle \nabla \renv{\phi}{f}{\lambda}(\y^t), \y - \y^t \big\rangle + \D[\frac{1}{\lambda}\phi + \omega](\y, \y^t).
$$
This illustrates a close relationship between alternating (Bregman) minimization and the (Bregman) proximal gradient method, which is known from the quadratic case; see e.g. \cite{lewis2007local,Och18,LWC18}. Indeed, in view of Proposition~\ref{prop:lsmad_env}, for $\phi$ super-coercive resp. $\dom \phi = \bR^m$ and $f$ relatively prox-bounded resp. relatively right prox-bounded with thresholds $\lambda_f$, both $(1/\lambda) \phi^* - \lenv{\phi}{f}{\lambda} \circ \nabla \phi^*$ and $(1/\lambda) \phi-\renv{\phi}{f}{\lambda}$ are proper lsc and convex if $\lambda \in \left]0,\lambda_f\right[$ and therefore locally satisfy the one-sided extended descent lemma with modulus $1/\lambda$ when $f$ is relatively prox-regular and $\lambda$ sufficiently small.
Overall, this means that existing convergence results from \cite{BBT2016,BSTV2018} for the Bregman proximal gradient method carry over, at least locally.

%Now suppose that the iterates $\w^t \to \w^*$ converge. Going back to the alternating minimization interpretation with $\sigma \equiv 0$, this means that also $\x^t \to \x^*:=\lprox{\phi}{f}{\lambda}(A(\w^t))$.

\section{Conclusions}
In this paper, we have considered the left and right Bregman proximal mapping of nonconvex functions including indicator functions of nonconvex sets. We define relative prox-regularity, an extension of prox-regurity, which provides us with a sufficient condition for the local single-valuedness of the left Bregman proximal mapping. In this context, we identify relatively amenable functions, i.e. compositions of a convex function and a smooth adaptable mapping as a main source for examples of relatively prox-regular functions.
Since the right Bregman proximal mapping can be related to the left Bregman proximal mapping via a substitution, many results can be transferred to the right Bregman proximal mapping. By way of example, we apply our theory to interpret joint alternating Bregman minimization with additional prox terms, locally, as Bregman proximal gradient.

%% file: incl_acknowledgements.tex
We would like to thank Tao Wu for fruitful discussions and helpful comments.

%% file: incl_appendix.tex
\section*{Appendix: Proof of First Part of Theorem~\ref{thm:convergence}} \label{sec:convergence_proof}
\addtocounter{section}{1}
\begin{lemma} \label{lem:sufficient_descent}
Let the assumptions in Theorem~\ref{thm:convergence} hold.
Then we have, for the iterates produced by Algorithm~\ref{alg:alternating_minimization_left_proximal}, that
\begin{enumerate}
\item[\rm (i)] A monotonic sufficient decrease over the iterates is guaranteed: 
\begin{align} \label{eq:sufficient_descent}
F_\lambda(\w^{t+1}, \x^{t+1}) +  \D[\sigma](\x^{t+1}, \x^t) + \D[\omega](\w^{t+1}, \w^t) \leq F_\lambda(\w^{t}, \x^{t}),
\end{align}
\item[\rm (ii)] $\{\w^t, \x^t\}_{t \in \bN}$ is bounded and $\x^t \in \intr(\dom \phi)$ for all $t$.
\item[\rm (iii)] We have that $-\infty < \beta\leq F_\lambda(\w^t, \x^t)$ is uniformly bounded from below for all $t$ and $\{F_\lambda(\w^{t+1}, \x^{t+1})\}_{t\in \bN}$ converges.
\end{enumerate}
\end{lemma}
\begin{proof}
In view of the coercivity of $F_\lambda$ and since $f,g$ are proper lsc, the iterates are well-defined.

For part (i) note that by the definition of the $\x$-update we have that
$$
F_\lambda(\w^t, \x^{t+1}) + \D[\sigma](\x^{t+1}, \x^t) \leq F_\lambda(\w^t, \x^t)
$$
and by the definition of the $\w$-update
$$
F_\lambda(\w^{t+1}, \x^{t+1}) + \D[\omega](\w^{t+1}, \w^t) \leq F_\lambda(\w^t, \x^{t+1}).
$$
Summing the two yields \eqref{eq:sufficient_descent}.

For part (ii) note that the boundedness of $\{\w^t, \x^t\}_{t \in \bN}$ follows from \eqref{eq:sufficient_descent} and the coercivity of $F_\lambda$. By the qualification condition and an argument similar to the one in the proof of Lemma~\ref{lem:prox_interior} we have that $\x^t \in \intr(\dom \phi)$.

For part (iii) note that $F_\lambda$ is proper and lsc and the iterates are bounded due to part (ii). In view of \cite[Corollary 1.10]{RoWe98}, $F_\lambda$ is bounded from below over the iterates and the conclusion follows.
\qed
\end{proof}

We are now ready to prove the statement from Theorem~\ref{thm:convergence}:
\begin{proof}
We sum the estimate \eqref{eq:sufficient_descent} form $t=0$ to $T$ and obtain, in view of Lemma~\ref{lem:sufficient_descent}(iii), that
\begin{align*}
-\infty < F_\lambda(\w^T, \x^T) - F_\lambda(\w^0, \x^0)&= \sum_{t=0}^T F_\lambda(\w^{t+1}, \x^{t+1}) - F_\lambda(\w^{t}, \x^{t}) \\
&\leq -\sum_{t=0}^T \big(\D[\sigma](\x^{t+1}, \x^t) +\D[\omega](\w^{t+1}, \w^t)\big).
\end{align*} 
We take $T \to \infty$ and deduce that
$$
\D[\sigma](\x^{t+1}, \x^t) +\D[\omega](\w^{t+1}, \w^t) \to 0,
$$
and therefore $\D[\sigma](\x^{t+1}, \x^t) \to 0$ and $\D[\omega](\w^{t+1}, \w^t) \to 0$ and in view of the strict convexity of $\sigma,\omega$ on $\intr(\dom \phi)$, we also have
$\|\x^{t+1}-\x^t\|\to 0$ and $\|\w^{t+1}-\w^t\|\to 0$.
In view of the $\x$- and $\w$-updates and the qualification condition \eqref{eq:qualification_interior}, we obtain that:
$$
0 \in \partial f(\x^{t+1}) + \frac{1}{\lambda} (\nabla \phi(\x^{t+1}) - A(\w^{t+1})) + \nabla \sigma(\x^{t+1}) - \nabla \sigma(\x^t) + \frac{1}{\lambda}(A(\w^{t+1}) - A(\w^t)), 
$$
and
$$
0 \in \partial g(\w^{t+1}) + \frac{1}{\lambda} A^* (\nabla \phi^*(A(\w^{t+1})) - \x^{t+1}) + \nabla \omega(\w^{t+1}) - \nabla \omega(\w^t).
$$
In view of \cite[Exercise 8.8(c)]{RoWe98} and \cite[Proposition 10.5]{RoWe98} and since $\x^{t+1} \in \intr(\dom \phi)$, this means
$$
\begin{pmatrix}
\nabla \sigma(\x^t) - \nabla \sigma(\x^{t+1}) + \frac{1}{\lambda}(A(\w^t) - A(\w^{t+1})) \\
\nabla \omega(\w^t) -\nabla \omega(\w^{t+1})
\end{pmatrix} \in \partial F_\lambda(\w^{t+1}, \x^{t+1}).
$$
In view of Lemma~\ref{lem:sufficient_descent}(ii), the iterates are bounded and we may consider a convergent subsequence $\{\w^{t_j}, \x^{t_j}\}_{j \in \bN} \subset \{\w^t, \x^t\}_{t \in \bN}$. Let $(\w^*, \x^*)$ denote the limit point. In view of the closedness of $\gph\partial F_\lambda$ under the $F_\lambda$-attentive topology, we have for $j \to \infty$, since \mynewline$F_\lambda(\w^{t_j}, \x^{t_j}) \to F_\lambda(\w^*, x^*)$, the continuity of $\nabla \sigma,\nabla \omega,A$ and $\|\x^{t+1}-\x^t\|\to 0$ and \mynewline$\|\w^{t+1}-\w^t\|\to 0$ that:
$$
0\in \partial F_\lambda(\w^*, \x^*).
$$
It remains to argue that also the limit point $\x^* \in \intr(\dom \phi)$ is contained in the interior of $\dom \phi$:
In view of the qualification condition \eqref{eq:qualification_interior} and an argument similar to the one in the proof of Lemma~\ref{lem:prox_interior}, as well as \cite[Proposition 10.5]{RoWe98}, we obtain that $\x^* \in \intr(\dom \phi)$ and conclude that the optimality conditions \eqref{eq:optimality_1} and \eqref{eq:optimality_2} hold.
\qed
\end{proof}